\documentclass[12pt]{amsart}
\usepackage[vmargin=2.5cm,hmargin=3cm]{geometry}
\usepackage{amssymb, latexsym, amsfonts,amsmath,amsthm, color,enumerate}
\usepackage{caption,float}
\def\@ptsize{2}
\title{Browkin's discriminator conjecture}
\author{Alexandru Ciolan}
\address{Rheinische Friedrich-Wilhelms-Universit\"at Bonn, Regina-Pacis-Weg 3, D-53113, Germany}
\email{calexandru92@yahoo.com}
\author{Pieter Moree}
\address{Max-Planck-Institut f\" ur Mathematik Bonn,
Vivatsgasse 7, D-53111 Bonn, Germany}
\email{moree@mpim-bonn.mpg.de}
\dedicatory{Dedicated to the memory of 
Prof. Jerzy Browkin (1934--2015)}
\newtheorem{Thm}{Theorem}
\newtheorem{Con}{Conjecture}
\newtheorem{Example}{Example}
\newtheorem{Prob}{Problem}
\newtheorem{Lem}{Lemma}

\newtheorem{Qu}{Question}
\newtheorem{Def}{Definition}

\newtheorem{cor}{Corollary}
\newcommand{\cal}{\mathcal}

\newtheorem{Prop}{Proposition}

\DeclareMathOperator{\lcm}{lcm}
\DeclareMathOperator{\ord}{ord}

\makeatletter
\let\@@pmod\pmod
\DeclareRobustCommand{\pmod}{\@ifstar\@pmods\@@pmod}
\def\@pmods#1{\mkern4mu({\operator@font mod}\mkern 6mu#1)}
\makeatother
\newcommand{\commA}[1]{\marginpar{%
\begin{color}{blue}
\vskip-\baselineskip 
\raggedright\footnotesize
\itshape\hrule \smallskip A: #1\par\smallskip\hrule\end{color}}}

\newcommand{\commP}[1]{\marginpar{%
\begin{color}{magenta}
\vskip-\baselineskip 
\raggedright\footnotesize
\itshape\hrule \smallskip P: #1\par\smallskip\hrule\end{color}}}

\begin{document}
\date{}
\maketitle

\begin{abstract}
\noindent Let $q\ge 5$ be a prime and put
$q^*=(-1)^{(q-1)/2}\cdot q$. We consider the integer sequence $u_q(1),u_q(2),\ldots,$ with $u_q(j)=(3^j-q^*(-1)^j)/4$. 
No term in this sequence is repeated
and thus for each $n$ there is a smallest integer $m$ such
that $u_q(1),\ldots,u_q(n)$ are pairwise incongruent modulo $m$. We write $D_q(n)=m$.
The idea of considering
the discriminator $D_q(n)$ is due to Browkin (2015) who, in case $3$ is a primitive root modulo $q,$
conjectured that the only values assumed by $D_q(n)$ are powers of $2$ and 
of $q$. We show that this is true for 
$n\neq 5$, but false for infinitely 
many $q$ in case $n=5$. We also determine $D_q(n)$ in case 
3 is not a primitive root modulo $q$.
\par Browkin's inspiration for 
his conjecture came from earlier
work of Moree and Zumalac\'arregui \cite{MZ}, who determined $D_5(n)$ for $n\ge 1$,
thus establishing a conjecture of S\u al\u ajan.
For a fixed prime $q$ their approach is easily 
generalized, but
requires some innovations in order to deal with all primes $q\ge 7$ 
and all $n\ge 1$. Interestingly enough, Fermat and Mirimanoff primes play a special role in this.
\end{abstract}
\section{Introduction}
Given a sequence of distinct positive integers $v=\left\lbrace v(j)\right\rbrace _{j=1}^{\infty}$ and any positive integer $n$, the discriminator $D(n)$ of the first $n$ terms of the sequence 
is defined as the smallest positive integer $m$ such that $v(1),\ldots,v(n)$ are pairwise incongruent modulo $m$. There are some results regarding the discriminator 
in case $v(j)$ is a polynomial in $j$ (see \cite{MZ} and references therein). 
Beyond the polynomial case, very little is known.
\par  In this paper we determine the discriminator
for the following infinite family of second order recurrences.
\begin{Def}
 Let $q\ge 5$ be a prime and put
$q^*=(-1)^{(q-1)/2}\cdot q$. 
The
sequence $u_q(1),u_q(2),\ldots,$ with 
$$u_q(j)=\frac{3^j-q^*(-1)^j}{4},$$
we call the Browkin-S{\u a}l{\u a}jan sequence for $q$.
\end{Def}
(In the sequel $p$ and $q$ will always denote primes.)
The sequence $u_q$ satisfies the
recursive relation $u_q(j)=2 u_q(j-1) + 3 u_q(j-2)$ for $j \geq 3,$ with initial 
values $ u_q(1)=(3+q^*)/4 $ and $u_q(2)=(9-q^*)/4$. 
In the context of the
discriminator, the sequence $u_5$ ($2, 1, 8, 19, 62, 181, 548, 1639, 4922,\ldots$) was first considered 
by Sabin S{\u a}l{\u a}jan during an internship carried out in 2012 at the Max Planck Institute for Mathematics in Bonn under the guidance of the second author.
A generalization of his sequence was introduced by 
Jerzy Browkin in an e-mail to the second author \cite{B}.
In the same e-mail, Browkin made the following conjecture.
\begin{Con}[Browkin, 2015] \label{conjecture}
Let $n\ge 1$ and let $q\ge 5$ be a prime such that
$3$ is a primitive root modulo $q$. Then $D_q(n)$ is
either a power of $2$ or a power of $q$.
\end{Con}
In formulating Conjecture \ref{conjecture}, Browkin was inspired by the following
result of  Moree and Zumalac\'arregui \cite{MZ} establishing
a conjecture made by S{\u a}l{\u a}jan.
\begin{Thm}[{\cite{MZ}}]
\label{main}
Let $n\ge 1$ be an arbitrary integer. Let $e$ be the smallest integer such that
$2^e\ge n$ and $f$ be the smallest integer such that $5^f\ge 5n/4$.
Then $$D_5(n)=\min\{2^e,5^f\}.$$
\end{Thm}
\noindent This result shows that Browkin's conjecture
holds true for $q=5$.
\par Our main result, 
Theorem \ref{Bmain1}, determines $D_q(n)$ for every $n\ge 1$ and prime $q\ge 5$. It provides, at the same time, the
first instance of the determination of the
discriminator for an infinite family of
second-order recurrences having characteristic
equation with rational roots. Very recently,
Faye, Luca and Moree \cite{FLM} determined the
discriminator for another infinite family, 
this time having irreducible characteristic
equation. Despite structural similarities, 
there are considerable differences in the details of the proofs in \cite{FLM} and
ours. For e.g., in our case it
is much harder to exclude small prime numbers as discriminator values. However, in the other 
case one has to work with elements and ideals in
quadratic number fields.
\par In number theory in general, and in our
paper in particular, the following primes 
play a special role.
\begin{Def} \medskip \hfil\break
A prime $q$ is said to be Artin if 
 $3$ is a primitive root modulo $q$.\hfil\break
 A prime $q$ is said to be Fermat if it is of the form $2^m+1$ with $m\ge 1$.\hfil\break
A prime number $q$ is said to be Mirimanoff if 
it satisfies $3^{q-1}\equiv 1\pmod*{q^2}$.
 \end{Def}
 \noindent The definition of Artin primes is non-standard
 and used here for brevity. 
 See Section \ref{speciaal} for more on 
 these special primes.
\par Our main result shows that 
Browkin's conjecture is true provided that
we exclude $n=5$. Theorem \ref{main} is obtained
on setting $q=5$.
We illustrate Theorem \ref{Bmain1} by examples in Section \ref{numericalresults}, Tables \ref{tableq=5}--\ref{tableq=29}.
\begin{Thm}
\label{Bmain1}
Let $q\ge 5$ be a prime and $n\ge 1$ an arbitrary 
integer. Then 
$$D_q(n)=
\begin{cases}
\min\{2^e,q^f:2^e\ge n,~q^f\ge \frac{q}{q-1}n\} & \text{if~}q\text{~is Artin, but not Mirimanoff;}\\
\min\{2^e,q:2^e\ge n,~q\ge n+1\} & \text{if~}q\text{~is Artin, Mirimanoff, but not Fermat};\\
\min\{2^e:2^e\ge n\} & \text{if~}q\text{~is Artin, Mirimanoff and Fermat};\\
\min\{2^e:2^e\ge n\} & \text{if~}q\text{~is not Artin},
\end{cases}$$
except for $n=5$ and
$q\equiv \pm 1\pmod*{28},$ in which case $D_q(5)=7.$
\par All the powers of $2$ and $q$ listed in each of the above 
subcases occur
as values, except that, in case $q$ is Artin 
but not Mirimanoff, then $q^f$ occurs if and only if
\begin{equation}
\label{fractionalinequality}
\Big\{f\frac{\log q}{\log 2}\Big\}>\frac{\log(q/(q-1))}{\log 2},
\end{equation}
where $\{x\}$ denotes the fractional part of
the real number $x$.
\end{Thm}
\begin{Example}
If $q$ is Artin, Mirimanoff, but not Fermat, then the powers
of $2$ and $q$ listed are $2^e$, $e\ge 0$ and $q$. All of them occur as values.
\end{Example}
Theorem \ref{Bmain1} suggests the following question.
\begin{Qu}
Does there exist a prime that is both Fermat and Mirimanoff?
\end{Qu}
If  such a prime exists, it is actually automatically
Artin by 
Lemma \ref{3Fermat}. However, finding it seems a chimera.
\par Taking into account the value of $D_q(5)$, we see
that the theorem leads us to partition the 
set of all primes $q\ge 5$ into eight subsets. 
These are considered in detail in Section 
\ref{primedistribution}, where they are 
listed with examples in Table 
\ref{eightcases}
 and
the natural density of each of them is 
(conditionally) evaluated and listed 
in Table \ref{densitiestable}. For example, the primes $q$ such
that $D_q(5)=7$ and $q$ 
is not Artin have, assuming the Generalized 
Riemann Hypothesis, the density $$\frac{5}{6}-\frac{173}{205}A=0.5177511101327382317\ldots,$$ where 
\begin{equation}
\label{Artinconstant}
A=\prod_{p\text{~prime}} \Big(1-\frac{1}{p(p-1)}\Big)=0.373955813619202288\ldots 
\end{equation}
is the Artin constant. By density
we mean the asymptotic ratio of the number of
primes up to $x$ in a set of primes and 
the total number of primes
up to $x$.
\begin{Def}
The set $\{D_q(n):n\ge 1\}$ we will denote by ${\mathcal D}_q$.
An integer $m\in {\cal D}_q$ is said to be a Browkin-S\u al\u ajan value, whereas
an integer $m\not\in {\cal D}_q$ is said to be a Browkin-S\u al\u ajan non-value. (For brevity, we sometimes use the shorter `value', respectively 
`non-value'.)
\end{Def}
The value of $q^*$ does not play a role in comparing terms
$u_q(i)$ and $u_q(j)$ with $i$ and $j$ of the same parity.
For this reason, various results
from \cite{MZ} can be copied (almost) verbatim. In such a case we
say that the proof follows by an equal parity index argument.
It is partly for this reason that the proof of our main result 
has a lot in common with the proof of
Theorem \ref{main} given in Moree and 
Zumalac\'arregui \cite{MZ}. Nevertheless, there are various complications
to be surmounted. In our proof we will show that, if $9\nmid d$, the
sequence is purely periodic with period $\rho_q(d)$. In \cite{MZ} 
the fact that $\rho_5(d)$ is even for $d>1$ plays a crucial role.
For general $q$ it
can happen that $\rho_q(d)$ is odd, and this is a source of complications. However,
luckily there is at most one exceptional $d$, namely $d=q$.\\
\indent For the convenience of the reader we 
prove our results in detail.
The extent to which the proofs are similar to the corresponding ones in 
\cite{MZ} is pointed out in the commentaries at the end of 
the sections. There are also results that have no counterpart 
in \cite{MZ}.

\section{Strategy of the proof
 of the main result}
As in \cite{MZ}, we think it is for the benefit of the reader to describe the strategy of the (now even lengthier) proof
of our main result, Theorem \ref{Bmain1}.
\par We start by showing that, if $n\le 2^e$, then
$D_q(n)\le 2^e,$ which will give us the crucial upper bound
$D_q(n)\le2n-1$.
\par We then study the periodicity of the sequence modulo $d$ and determine
its period $\rho_q(d)$. The information obtained by doing so will be used to exclude
many values of $d$ from being Browkin-S\u al\u ajan.
In case $9\nmid d,$ the sequence turns out to be (purely) periodic, with a
period that we can compute exactly. As
we can easily see that $3\nmid D_q(n),$ this will be enough to serve our purposes. 
\par Restricting our attention now to those $d$ for which $9\nmid d$ and
using that $D_q(n)<2n,$ we see that, if $\rho_q(d)\le d/2$, then $d$ is a  
Browkin-S\u al\u ajan non-value. The basic property (\ref{per2}) of the period, together with an analysis of its parity,
will then exclude composite values of $d$. Thus we must have $d=p^m$, with $p$ a prime.
\par In order for $\rho_q(p^m)>p^m/2$ to hold, we find that we must have $\ord_p(9)=(p-1)/2$, that
is, $9$ must have maximal possible order modulo $p$. 
The set of these primes $p\ge 5$ different from (any fixed)
$q$ is denoted by $\cal P$ and will play an important
role.
In fact, $9$ must have maximal possible order modulo $p^m$, that is, $\ord_p(9)=\varphi(p^m)/2,$ for any $m\ge1.$
(A square cannot have a multiplicative order
larger than $\varphi(p^m)/2$ modulo $p^m$.) 
This is about as far as the study of the periodicity will
get us. To get further we will use a more refined tool, the {\it incongruence index}, which, for any given integer $m,$
is the largest integer $k$ such that $u_q(1),\ldots,u_q(k)$ are pairwise distinct modulo $m.$ We
denote this by $\iota_q(m)=k$. It is easy to see that $\iota_q(d)\le \rho_q(d)$ if the sequence is purely periodic 
modulo $d.$ Using again that $D_q(n)<2n,$ one 
notes that, similarly with the period, if $\iota_q(d)\le d/2$, then $d$ is a Browkin-S\u al\u ajan non-value. 

For the primes $p>3$ 
we show by a lifting argument that, if $\iota_q(p)<\rho_q(p)$, 
then $p^2,p^3,\ldots$ are Browkin-S\u al\u ajan non-values.
Likewise, we prove that, if $\iota_q(p)\le p/2$, then $p,p^2,p^3,\ldots$ are 
Browkin-S\u al\u ajan non-values.
We then show that all primes in $\cal P$ satisfy $\iota_q(p)<\rho_q(p)$.
\par At this point, for any fixed prime $q,$ we are left with the primes $p$ in $\cal P$  as the only 
possible Browkin-S\u al\u ajan values.
Then, using classical combinatorial number theory techniques, 
we infer that no prime 
$p>2060$ different from $q$ can be a Browkin-S\u al\u ajan value. In order to deal with the remaining
primes in $\cal P$, we study the quantity $\upsilon(p)=\max\{\iota_q(p):q\ge 5,~q\ne p\}$, which we
dub the {\it{universal incongruence index.}} It is easy to see that,
if $\upsilon(p)\le (p+1)/2,$ then $D_q(n)\ne p$ for $p\ne q$.
We provide a simple way to compute $\upsilon(p)$ and use
this to check that the inequality 
$\upsilon(p)\le (p+1)/2$ holds for $29<p<2060.$
By a slightly more refined approach we manage to
show that, in fact, $p=7$ is the only 
prime $p\ne q$ that can arise as Browkin-S\u al\u ajan value; it can be seen directly for which values of $n$ and $q$ it occurs. 
\par Apart from this exception, we are left with $D_q(n)=2^e$ for some $e$ or
$D_q(n)=q^f$ for some $f$. The first case
is trivial. In the analysis of the second case, Artin, Fermat and Mirimanoff primes naturally
appear. For instance, if $q$ is Mirimanoff, then 
powers $q^f$ with $f\ge 2$ can not appear as values, whilst $q$ does.
This analysis is not 
complicated, but rather long-winding and
therefore we will not say more about it until later.

\section{Preparations for the proof}
\label{preparations}

\subsection{The sequence $u_q$ viewed as an interlacing}
\label{interlacing}
The sequence $u_q$ can be regarded as an interlacing of the sequence $u_{1,q}$ 
consisting of the odd indexed elements and the sequence  $u_{2,q}$ consisting of 
the even indexed elements. We have
$$u_{1,q}(n)=u_q(2n-1)=(3^{2n-1}+q^*)/4,~
u_{2,q}(n)=u_q(2n)=(3^{2n}-q^*)/4.$$ 
In order to determine whether a given $m$ discriminates $u_q(1),\ldots,u_q(n)$
modulo $m$, we separately consider whether $u_q(i)\equiv u_q(j)$ modulo $m$, 
with $i$ and $j$ of the same parity and with $i$ and $j$ of different parity.
In the first case, the behavior modulo $m$ is determined by that of consecutive
powers of 9 modulo $m$. 
\begin{Lem}
\label{oneventje}
Suppose that $3\nmid m$ and $1\le \alpha\le n$. We have $u_q(i)\not\equiv u_q(j)\pmod* m$ for
every pair $(i,j)$ satisfying $\alpha\le i< j\le n$ with $i\equiv j\pmod* 2$
if and only if $\ord_{4m}(9)>(n-\alpha)/2$.
\end{Lem} 
\begin{proof} We have $u_q(i)\not\equiv u_q(i+2k)\pmod* m$ iff $9^k\not\equiv 1\pmod*{4m}$. Thus
$u_q(i)\not\equiv u_q(j)\pmod* m$ for every pair $(i,j)$ with $\alpha\le i<j\le n$ and
$i\equiv j\pmod*2$ iff $9^k\not\equiv 1\pmod*{4m}$ for $1\le k\le (n-\alpha)/2$. \end{proof}

\noindent {\tt Commentary}. Lemma \ref{oneventje} is proved by an equal
index parity argument.

\subsection{The sequence $u_q$ modulo powers of $2$}
We will show that $2^e$ with $2^e\ge n$ discriminates $u_q(1),\ldots,u_q(n)$, that is,
we will show that, if $2^e\ge n,$ the terms of the sequence
$u_q(1),\dots, u_q(n)$ lie in distinct residue classes modulo
$2^e$.

Let $p$ be a prime. If $p^a|n$ and $p^{a+1}\nmid n$, then we
put $\nu_p(n)=a$. The following result is well-known; for a proof see,
e.g., Beyl \cite{Beyl}.
\begin{Lem}
\label{Beyl}
Let $p$ be a prime, $r\ne -1$ an integer satisfying $r\equiv 1\pmod*p$ and $n$ a natural number.
Then
$$
\nu_p(r^n-1) =
\begin{cases}
\nu_2(n)+\nu_2(r^2-1)-1 & \text{if $p=2$ and $n$ is even};\\
\nu_p(n)+\nu_p(r-1) & \text{otherwise}.
\end{cases}
$$
\end{Lem}



A crucial fact about $u_q$ is that its terms have alternating parity.
Indeed, we have the
following trivial observation (note that $q^*\equiv 1\pmod*4$).
\begin{Lem}
\label{trivialparity}
If $q^*\equiv 1\pmod*8,$ the terms of $u_q(1),u_q(2),\ldots$ alternate between
odd and even. If $q^*\equiv 5\pmod*8,$ it is the other way around.
\end{Lem}

Armed with these two lemmas we are ready to establish the following result.

\begin{Lem}
\label{prop2n}
Let $n \geq 2$ be an integer with $n\le 2^m$.  Then $u_q(1),\dots,u_q(n)$ are pairwise distinct modulo $2^m$.
\end{Lem}
\begin{proof} For $n=2$ the result is obvious by 
Lemma \ref{trivialparity}. So assume that $n\ge 3$.
Since by 
Lemma \ref{trivialparity} the terms of the sequence alternate in parity, it suffices to compare the remainders $\pmod*{2^m}$ of
the terms having an index with the same parity. Thus assume that we have
$$u_q(2j+\alpha) \equiv u_q(2k+\alpha)\pmod*{2^m}{\rm ~with~}1\le 2j+\alpha<2k+\alpha \leq n,~\alpha\in \{1,2\}.$$ It follows from this that $9^{k-j}\equiv 1 \pmod*{2^{m+2}}$.
We have $\nu_2(9^{k-j}-1)=\nu_2(k-j)+3$ by Lemma \ref{Beyl}. Further, 
$2k-2j \leq n-1<2^m$, so $\nu_2(k-j) \leq m-2$ (here we used that $n\ge 3$).
Therefore  $\nu_2(9^{k-j}-1)=\nu_2(k-j)+3 \leq (m-2) +3 = m+1$, which implies that 
$9^{k-j}\not \equiv 1 \pmod*{2^{m+2}}$. Contradiction.\end{proof}

On noting that trivially $D_q(n)\ge n$ and that,  
for $n\ge 2,$ the interval $[n,2n-1]$ always contains some power of $2$, we
obtain the following corollary to Lemma \ref{prop2n}.
\begin{cor}
\label{bertie}
We have $n\leq D_q(n)\le 2n-1$.
\end{cor}

\noindent {\tt Commentary}. Lemma \ref{prop2n} is proved by an equal
index parity argument.

\section{Periodicity and discriminators}
\subsection{Generalities}
\label{general}
We say that a sequence of integers $\{v_j\}_{j=1}^{\infty}$ is {\it (eventually) periodic} 
modulo $d$ if there exist integers $n_0\ge 1$ and $k\ge 1$ such that
\begin{equation}
\label{per1}
v_n\equiv v_{n+k}\pmod*d
\end{equation}
for every $n\ge n_0$. 
The minimal choice for $n_0$ is called the {\it pre-period}.
The smallest $k\ge 1$ for which (\ref{per1}) holds
for every $n\ge n_0$ is said to be the {\it period} and denoted by $\rho_v(d)$.
In case we can take $n_0=1$ we
say that the sequence is {\it purely periodic} modulo $d$.

\smallskip

 Let $\{v_j\}_{j=1}^{\infty}$ be a second order linear recurrence with the 
two starting values and the coefficients of the defining equation being integers.
Note that, for a given $d$, there must be a pair $(a,b)$ such that $a\equiv v_n$ and $b\equiv v_{n+1}$ modulo $d$ 
for infinitely many $n$. Since a pair of consecutive terms determines uniquely all the subsequent ones,
it follows that the sequence is periodic modulo $d$. If we consider $n$-tuples instead of pairs modulo $d$, we
see that an $n$th order linear recurrence with the 
$n$ starting values and the coefficients of the defining equation  being integers is always periodic
modulo $d$.\\
\indent If a sequence $v$ is periodic modulo $d_1$ and modulo $d_2$ with
$(d_1,d_2)=1$, then we obviously have
\begin{equation}
\label{per2}
\rho_v(d_1d_2)=\lcm(\rho_v(d_1),\rho_v(d_2)).
\end{equation}
If the sequence is purely periodic modulo $d_1$ and modulo $d_2$ with
$(d_1,d_2)=1$, then it is also purely periodic modulo $d_1d_2$.
Another trivial property of $\rho_v$ is that, if 
the sequence $v$ is periodic modulo $d_2$, then for
every divisor $d_1$ of $d_2$ we have
\begin{equation}
\label{per3}
\rho_v(d_1)|\rho_v(d_2).
\end{equation}
\indent The following result links the period with the discriminator. Its moral is that,
if $\rho_v(d)$ is small enough, we cannot expect $d$ to occur as $D_v$-value, i.e., $d$ does not 
belong to the image of $D_v$.

\begin{Lem}
\label{gee}
Assume that $D_v(n)\le g(n)$ for every $n\ge 1$ with $g$ non-decreasing.
Assume that the sequence $v$ is purely periodic modulo $d$ with period $\rho_v(d)$. If $g(\rho_v(d))<d$, then $d$ is a $D_v$-non-value.
\end{Lem}
\begin{proof} Since $v_1\equiv v_{1+\rho_v(d)}\pmod*d$ we must have $\rho_v(d)\ge n$.
Suppose that $d$ is a $D_v$-value, that is, for some $n$ we have $D_v(n)=d$. Then
$d=D_v(n)\le g(n)\le g(\rho_v(d))<d$. Contradiction. \end{proof}

\noindent {\tt Commentary}. This section is taken over verbatim from \cite[Section 4]{MZ}.

\subsection{Periodicity of the Browkin-S\u al\u ajan sequence}
\label{peri}
The purpose of this section is to establish Theorem~\ref{generalperiod}, which gives an explicit formula for the 
period $\rho_q(d)$ and the pre-period for the Browkin-S\u al\u ajan sequence.
Since it
is easy to show that $3\nmid D_q(n)$, it would be actually enough to study 
those integers $d$ with $3\nmid d$ (in which case the 
Browkin-S\u al\u ajan sequence is purely 
periodic modulo $d$). However, for completeness, we discuss the periodicity of
the Browkin-S\u al\u ajan sequence for {\it every} $d$. In the sequel it is helpful
to have in mind the trivial observation 
that, if $3\nmid m$, then 
\begin{equation}
\label{lcmorder}
2\ord_{4m}(9)=\lcm(2,\ord_{4m}(3)).
\end{equation}
\begin{Thm}
\label{generalperiod}
Suppose that $d>1$. Write $d=3^{\alpha}\cdot \delta$ with $(\delta,3)=1$. 
The period $\rho_q(d)$ of the Browkin-S\u al\u ajan sequence modulo $d$ exists
and satisfies
$$
\rho_q(d)=
\begin{cases}
\ord_{4\delta}(9) &\mbox{if }d=q \mbox{ and }2\nmid \ord_q(3);\cr
2\ord_{4\delta}(9) & \mbox{otherwise}.
\end{cases}
$$
The pre-period equals
$\max(1,\alpha)$.
\end{Thm}
\begin{cor}
\label{9}
The Browkin-S\u al\u ajan sequence is purely periodic
if and only if $9\nmid d$. 
\end{cor}
In the proof of the next lemma we will use that $\rho_q(3)=2$.
Since modulo $3^f$ the sequence $u_q$ eventually alternates 
between $q^*/4$ and $-q^*/4,$ it even follows that $\rho_q(3^f)=2$ for
every $f\ge 1$.
\begin{Lem}
\label{withtypo}
Write $d=3^{\alpha}\cdot \delta$ with $(\delta,3)=1$. The 
Browkin-S\u al\u ajan sequence is purely periodic
if and only if $9\nmid d$. Furthermore, if $9\nmid d$, then 
$$
\rho_q(d)=
\begin{cases}
2\ord_{4\delta}(9) & \text{if~}2\mid \rho(d);\cr
\ord_{4\delta}(9) & otherwise.
\end{cases}
$$
\end{Lem}
\begin{proof} Since $u=(3+q^*)/4,{\overline{q^*/4,-q^*/4}}\pmod*9$ 
and $3+q^*\not\equiv -q^*\pmod*9,$
the condition $9\nmid d$ is 
necessary for the Browkin-S\u al\u ajan sequence to be purely periodic modulo $d$. 

We will now show that it is also sufficient.
Using that $\rho_q(3)=2,$ it follows that
$u_q(n)\equiv u_q(n+2k)\pmod*d$ iff $3^n\equiv 3^{n+2k}\pmod*{4\delta}$. 
Since there exists $k$ satisfying $3^{2k}\equiv 1\pmod*{4\delta},$ it follows
that the Browkin-S\u al\u ajan sequence is purely periodic modulo $d$. 
Moreover, we have  $\rho_q(d)\,|2\ord_{4\delta}(9)$ 
with equality if $\rho_q(d)$ is even and $\rho_q(d)=\ord_{4\delta}(9)$ 
otherwise.\end{proof}
\begin{cor}
\label{lcmlcm}
We have $2\ord_{4\delta}(9)=\lcm(2,\rho_q(d))$.
\end{cor}

\noindent We next determine the parity of $\rho_q(d)$.
\begin{Lem}
\label{parity}
Suppose that $9\nmid d$ and $d>1$. We have that
$2\nmid \rho_q(d)$ if and only if $d=q$ and $2\nmid \ord_q(3)$.
\end{Lem}
\begin{proof}
Suppose that $d$ satisfies the conditions of the lemma and 
$\rho_q(d)$ is odd. Then $d>2$.
Since $\rho_q(3)=2$ it follows that $(d,3)=1$.
We have
\begin{equation}
\label{per4}
u_q(n)\equiv u_q(n+\rho_q(d))\pmod*d
\end{equation}
iff
\begin{equation}
\label{per5}
(1-3^{\rho_q(d)})/2\equiv q^*(-3)^{-n}\pmod*{2d}.
\end{equation}
{\sc{Case 1.}} $(q,d)=1$. If (\ref{per5}) is to hold for every $n\ge 1$, then
$(-3)^n$ assumes only one value as $n$ ranges 
over the positive integers. Since $(-3)^{\varphi(2d)}\equiv 1\pmod*{2d}$ we must have
$(-3)^n\equiv 1\pmod*{2d}$ for every $n\ge 1$. 
This has no solution with $d>2$.\\
{\sc{Case 2.}} $q|d$. If the left-hand side of (\ref{per5}) is not divisible by $q$, then (\ref{per5})
has no solution and $\rho_q(d)$ must be even. So assume that the
left-hand side is divisible by $q$.
Then 
\begin{equation}
(1-3^{\rho_q(d)})/(2q)\equiv (-1)^{(q-1)/2}(-3)^{-n}\pmod*{2d/q}.
\end{equation}
The only possible solutions here are $d=q$ and $d=2q$. Since $\rho_q(2)=2$ 
we are left with $d=q$. 
Since $u_q(j)\equiv 3^j/4\pmod*q$ we infer that $\rho_q(q)=\ord_q(3)$,
which is odd iff $\ord_q(3)$ is odd.\end{proof}

\begin{proof}[Proof of Theorem \ref{generalperiod}] It is an easy observation that
modulo $3^{\alpha}$ the Browkin-S\u al\u ajan sequence has pre-period $\max(\alpha,1)$ 
and period two. {}From (\ref{generalperiod}) we infer 
that $\rho_q(q)=\ord_q(3)=\ord_{4q}(9)$ if $\rho_q(q)$ is odd.
This then, in combination with Lemmas \ref{withtypo} and \ref{parity}, 
completes the proof. \end{proof}

\noindent {\tt Commentary}. Since now $\rho_q(d)$ can be odd, various complications arise and 
the results become a bit more difficult to formulate. The proofs
proceed, however, in the same way as before. Lemmas \ref{withtypo} and
\ref{parity} taken together cover the same material as
Lemmas 6 and 7 of \cite{MZ}; however, we think we improved the presentation. 
In Lemma \ref{withtypo} we determine the order in case it is even. In Lemma \ref{parity}
we determine all cases where the order is odd. This is more logically structured than
in \cite{MZ}. In the earlier version, Theorem 2 and Lemma 7 cover practically the same
ground; this is avoided in the new version.
Our new approach also avoids having to make the case distinction between
$\alpha=0$ and $\alpha=1$.\\
\indent We would also like to point out that, instead of $\ord_{\delta}(9)$ in \cite[Lemma 6]{MZ}, one
should read $\ord_{4\delta}(9)$ (but that is also clear from the proof given in \cite{MZ}).

\subsection{Comparison of $\rho(d)$ with $d$}
For notational convenience, we will from now on use $\rho(d)$ instead of $\rho_q(d),$ unless the dependence on $q$ is necessary to be pointed out.
\begin{Lem}
\label{previous}
We have $\rho(2^e)=2^e$ and $\rho(3^e)=2$. If $p$ is odd, then $\rho(p^e)|\varphi(p^e)$. 
\end{Lem}
\begin{proof} From Lemma \ref{Beyl} it follows 
that $\ord_{2^{e+2}}(9)=2^{e-1}$ and hence, by Theorem \ref{generalperiod}, $\rho(2^e)=2^e$. 
For $n$ large enough, modulo $3^e$ the sequence alternates between $-q^*/4$ and 
$q^*/4$ modulo $3^e$.
Since these are different residue classes, we have $\rho(3^e)=2$.\\
\indent It remains to prove the final claim. 
If $p=3$ it is clearly true and thus we may assume that $p>3$. 
Note that $\ord_{4p^e}(9)=\ord_{p^e}(9)$ and thus it follows from Lemma \ref{withtypo}
that
$\rho(p^e)|2\ord_{p^e}(9)\,|\,2(\varphi(p^e)/2)$. \end{proof}

This lemma together with Theorem \ref{generalperiod} and 
(\ref{per2}) yields the following result.
\begin{Lem}
\label{previous1}
We have $\rho(d)\le \lcm(2,\rho(d))\le d$.
\end{Lem}
The sharper bound $\rho(m)\le m/2$ holds in case $m$ is not a prime power.
\begin{Lem}
\label{uppierho}
Suppose that $d_1,d_2>1$ and $(d_1,d_2)=1$. Then
$$\rho(d_1d_2)\le d_1d_2/2.$$
\end{Lem}
\begin{proof} We have $\rho(d_1d_2)=\lcm(\rho(d_1),\rho(d_2))$. 
In case both $\rho(d_1)$ and $\rho(d_2)$ are even it thus follows that
$\rho(d_1d_2)\le \rho(d_1)\rho(d_2)/2$. 
Lemma \ref{previous1} then gives $\rho(d_1d_2)\le d_1 d_2/2$.
By Lemma \ref{parity} it remains, without loss of generality, to deal with the case
where $d_1=q$ and $\rho(q)=\ord_q(3)$ is odd. Since $\varphi(q)$
is even, we see that 
$\rho(q)\le (q-1)/2$. We then infer that
$\rho(qd_2)\le \rho(q)\rho(d_2)\le q \rho(d_2)/2\le qd_2/2$. \end{proof}

\noindent {\tt Commentary}. The fact that the period can be odd requires
some modifications. In Lemma \ref{previous} we use that 
$\rho(p^e)|2\ord_{p^e}(9)$ instead of $\rho_5(p^e)=2\ord_{p^e}(9)$. 
In the proof of Lemma \ref{uppierho} we have to deal with the case $d_1=q$ separately.

\section{Browkin-S\u al\u ajan non-values of $D_q$}
\label{non-values}
Recall that, if $m=D_q(n)$ for some $n\ge 1,$ we call $m$ a 
Browkin-S\u al\u ajan value and
otherwise a Browkin-S\u al\u ajan non-value.
\par Most of the following proofs rely on the simple fact that for certain
sets of integers we have that, if $u_q(1),\dots,u_q(n)$ are in 
$n$ distinct residue classes modulo $m$, then $m\ge 2n$ 
contradicting Corollary \ref{bertie}.

\subsection{$D_q(n)$ is not a multiple of $3$}
\begin{Lem}
\label{notdrie}
We have $3\nmid D_q(n)$.
\end{Lem}
\begin{proof} We argue by contradiction and so assume that 
$D_q(n)=3^{\alpha}m$ with $(m,3)=1$ and $\alpha \ge 1$.
Since by definition $u_q(\alpha)\not\equiv u_q(\alpha+2t)\pmod*{3^{\alpha}m}$
for $t=1,\ldots,\lfloor{(n-\alpha)/2}\rfloor$ and  $u_q(\alpha)\equiv 
u_q(\alpha+2t)\pmod*{3^{\alpha}}$ 
for every $t\ge 1$, it follows that $u_q(i)\not\equiv u_q(j)\pmod*m$ 
with $\alpha \le i<j\le n$ and $i$ and $j$ of the same parity. By Lemma \ref{oneventje}
it then follows that $\ord_{4m}(9)>(n-\alpha)/2$. By (in this order) 
Corollary \ref{lcmlcm}, Lemma \ref{previous1}
and Corollary \ref{bertie} we then find that $n-\alpha+1\le 2\ord_{4m}(9)=
\lcm(2,\rho(m))\le m\le 2n/3^{\alpha}$. 
This implies
that $n\le 3^{\alpha}(\alpha-1)/(3^{\alpha}-2)$. On the other hand, by
 Corollary \ref{bertie} we have
$3^{\alpha}m\le 2n$ and hence $n\ge 3^{\alpha}/2$. Combining the upper and the
lower bound for $n$ yields
$3^{\alpha}\le 2\alpha$, which has no solution with $\alpha\ge 1$. \end{proof}

\noindent {\tt Commentary}. 
This proof is quite similar to that of \cite[Lemma 11]{MZ}.
The $2\ord_{4m}(9)=\rho(m)$ there has now been replaced by
the identity $2\ord_{4m}(9)=\lcm(2,\rho(m))$ (Corollary 
\ref{lcmlcm}). Instead of the earlier $\rho(m)\le m,$
we now need $\lcm(2,\rho(m))\le m$, but this is true by Lemma \ref{previous1}.
In the earlier proof there is $n\le 3^{\alpha}(\alpha-1)/(3^{\alpha}-1)$
instead of the correct $n\le 3^{\alpha}(\alpha-1)/(3^{\alpha}-2)$ and
$3^{\alpha}\le 2\alpha-1$ instead of the correct $3^{\alpha}\le 2\alpha$.

\subsection{$D_q(n)$ is a prime-power}
\begin{Lem}\label{nonnie}
Suppose that $d$ with $9\nmid d$ satisfies $\rho(d)\le d/2.$
Then $d$ is a Browkin-S\u al\u ajan non-value.
\end{Lem}
\begin{proof} Suppose that $d=D_q(n)$ for some integer $n$. 
By Lemma \ref{generalperiod} the condition $9\nmid d$ guarantees that 
the Browkin-S\u al\u ajan sequence is purely
periodic modulo $d$. 
Therefore $u_q(1)\equiv u_q(1+\rho(d))\pmod*d$ and so
$\rho(d)\ge n$. The assumption 
$\rho(d)\le d/2$  now implies that 
$d\ge 2\rho(d)\ge 2n$, contradicting
Corollary \ref{bertie}. \end{proof}

\indent We now have the necessary ingredients to establish the following result. 
Let $p$ be odd. On noting that in $(\mathbb Z/p^m\mathbb Z)^*$ a square has maximal order $\varphi(p^m)/2$, we
see that the following result says that a 
Browkin-S\u al\u ajan value is either a power of two or a 
prime power $p^m$ with $9$ having maximal multiplicative order in $(\mathbb Z/p^m\mathbb Z)^*$.
\begin{Lem}
\label{redpp}
A Browkin-S\u al\u ajan value greater than $1$ must be of the form $p^m$, with $p=2$ or $p>3$ 
and $m\ge 1$. Further, one must have 
$\ord_{p^m}(9)=\varphi(p^m)/2$ and $\ord_p(9)=(p-1)/2$. If $m\ge 2$, 
then $p$ is not Mirimanoff. In case $p^m=q$ we even have
that $\ord_q(3)=q-1$.
\end{Lem}
\begin{proof} Suppose that $d>1$ is a Browkin-S\u al\u ajan value that is not a prime power. 
Thus we can write $d=d_1d_2$ with
$d_1,d_2>1$, $(d_1,d_2)=1$. By Lemma \ref{notdrie} we have  $3\nmid d_1d_2$. 
By Lemma \ref{uppierho} we have $\rho(d_1d_2)\le d_1d_2/2$, which by
Lemma \ref{nonnie} implies that $d=d_1d_2$ is a non-value. Thus $d$ is a 
prime power $p^m$. By Lemma \ref{notdrie} we have $p=2$ or $p>3$. Now
let us assume that $p>3$. 
By Lemma \ref{previous} we have either $\rho(p^m)=\varphi(p^m)$ or $\rho(p^m)\le \varphi(p^m)/2$.
The latter inequality leads to $\rho(p^m)\le p^m/2$ and hence to $p^m$ being a non-value. 
Thus we must have $\rho(p^m)=\varphi(p^m)$.
\medskip\hfil\break
{\sc{Case 1.}} $p^m=q$. Here we note that $\rho(q)=\ord_q(3)$. If 3 is not a primitive
root modulo $q$, then $\rho(q)|\varphi(q)/2$ and hence is $\le q/2$ and so a non-value.
Thus $\ord_q(3)=q-1$ and hence $\ord_q(9)=(q-1)/2$.\medskip\hfil\break
{\sc{Case 2.}} $p^m\ne q$. The condition $\rho(p^m)=\varphi(p^m)$ by
Theorem \ref{generalperiod} can be rewritten as
 $\ord_{p^m}(9)=\varphi(p^m)/2$. Now, 
if $\ord_p(9)<(p-1)/2$, this leads
to $\ord_{p^m}(9)<\varphi(p^m)/2$ and hence we must have $\ord_p(9)=(p-1)/2$. 
Finally, suppose that $m\ge 2$ and 
that $p$ is Mirimanoff, that is, $3^{p-1}\equiv 1\pmod*{p^2}$. 
Then $\ord_{p^m}(9)\le \varphi(p^m)/p<\varphi(p^m)/2$. This contradiction shows that, if $m\ge 2$, then $p$
is not Mirimanoff. \end{proof}

\noindent {\tt Commentary}. The statement and proof of
Lemma \ref{redpp} is similar to that of 
\cite[Lemma 13]{MZ}, but with the case $p^m=q$ being
considered separately.

\subsection{Powers of $q$ assumed by $D_q(n)$}
In the study of ${\cal D}_q$ the powers of $q$ play a special role and require separate
consideration. We will use the following simple result.
\begin{Lem}
\label{valuenotFermat}
Let $p$ be a Fermat prime. Then $p\not\in \cup_{q\ge5}{\cal D}_q$.
\end{Lem}
\begin{proof} By contradiction. So suppose that $D_q(n)=p$ for some $q\ge 5$ and $n\ge 1$.
Write $p=2^{e}+1$. By Lemma \ref{prop2n} we see that $D_q(n)\le 2^{e}$
for $n\le p-1$. Since $D_q(n)\ge n>p$ for $n>p$, it follows that $n=p$. As
$u_q(1)\equiv u_q(p)\pmod*p$, this is impossible. \end{proof}

\begin{Lem}
\label{prop5n} Let $q\ge 5$ be a prime.\\
{\rm a)} If $q$ is Artin, then the integers $u_q(1),\dots,u_q(n)$ are pairwise distinct modulo 
$q$ if and only if $q\ge n+1.$\\
{\rm b)} If $q$ is Artin and not Mirimanoff,
 then the integers $u_q(1),\dots,u_q(n)$ are pairwise distinct modulo 
$q^{f}$ if and only if $$q^f\ge \frac{qn}{q-1}.$$
\end{Lem}
\begin{proof} If $q^f<qn/(q-1)$, then $1+(q-1)q^{f-1}\le n$. By Lemma \ref{Beyl} we have
$$3^{(q-1)q^f}\equiv 1\pmod*{q^f},$$ which ensures that 
$u_q(1)\equiv u_q(1+(q-1)q^f)\pmod*{q^f}$. 
Thus the
condition $q^f\ge qn/(n-1)$ is necessary in order
that $u_q(1),\ldots,u_q(n)$ are pairwise distinct modulo $q^f$. We next show it is also sufficient.
So assume that 
$q^f\ge qn/(q-1)$. We distinguish the following two cases.\medskip\hfil\break
a) We let $f=1$ and we have to show that $u_q(1),\ldots,u_q(n)$ are pairwise distinct
modulo $q$ iff $q\ge n+1$. This is a consequence of the sequence being periodic
with period $q-1$ (as by assumption 
$q$ is an Artin prime).\medskip\hfil\break
b) The statement in case $f=1$ is a weaker
version of part a). So we may assume that $f\ge 2$.
It suffices to show that $u_q(j_1)\not\equiv u_q(k_1)\pmod*{q^f}$ with $1\le j_1< k_1\le n$ in the same congruence class modulo $q-1$. We will argue by contradiction
and so assume that
$$u_q((q-1)j+\alpha) \equiv u_q((q-1)k+\alpha)\pmod*{q^{f}},$$ with 
$1\le (q-1)j+\alpha<(q-1)k+\alpha \leq n$ and $1\le \alpha \le q-1$. From this it follows that $$3^{(q-1)(k-j)}\equiv 1 \pmod*{q^{f}},$$ where
$$k-j\le \frac{n-\alpha}{q-1} <\frac{n}{q-1}\le q^{f-1}$$ by hypothesis and hence 
$\nu_q(k-j)\le f-2$. The assumption that $q$ is not Mirimanoff prime ensures
that $\nu_q(3^{q-1}-1)=1$.
On invoking Lemma \ref{Beyl} we now infer that 
$$\nu_q(3^{(q-1)(k-j)}-1)=\nu_q(k-j) + \nu_q(3^{q-1}-1) \leq f-2+1=f-1,$$ 
contradiction. \end{proof}

The following result allows one to
determine  precisely which powers of
$q$ appear in ${\cal D}_q$.
\label{valuesets}
\begin{Lem}
\label{qpowerspresent}
Let $q\ge 5$ be a prime.\\
{\rm a)} We have $q\in {\cal D}_q$ if and only if 
$q$ is Artin and $q$ is not 
Fermat.\\
{\rm b)} Let $f\ge 2$. 
Then $q^f\in {\cal D}_q$ 
if and only if $q$ is Artin and not Mirimanoff, and
satisfies, for some natural number $e,$ the
inequality
\begin{equation}
\label{ciolan}
{q\over q-1}(2^e+1)\le q^f<2^{e+1}.
\end{equation}
\end{Lem}
\begin{proof}\medskip \hfil\break  a)
Note that $u_q(j)\equiv 3^j/4\pmod*p$.
Thus, if $q$ is Artin, then  $u_q(1),\ldots,u_q(q-1)$ 
are pairwise distinct modulo $q$ and hence
$D_q(q-1)\le q$.
If $q$ is not Fermat, then $q-1$ is not a power of 
any prime number $p\ge 2$
and so by Lemma \ref{redpp} it
follows that $D_q(q-1)=q$.
If $q$ is Fermat, 
then $q\not \in {\cal D}_q$ by Lemma \ref{valuenotFermat}. If $q$ is 
not Artin, then $\rho(q)\le q/2$ and $q\not \in {\cal D}_q$ by
Lemma \ref{nonnie}.\medskip \hfil\break
\noindent b) Let $f\ge 2$ 
and $q^f\in\cal D_q.$ 
Note that $\rho(q)=\ord_q(3)$. If $q$ is not
Mirimanoff, then $\rho(q^f)\le q^f/q<q^f/2$ and $q^f$
is a non-value. If $q$ is not Artin, then 
$\rho(q^f)\le q^{f-1}\rho(q)\le q^f/2$ and again $q^f$
is a non-value.
\par Now assume that $q$ is Artin and not Mirimanoff.
By Lemma \ref{prop5n}, if $q^f=D_q(n),$ then $$q^f\ge \frac{qn}{q-1}.$$ Therefore, assuming
$$\frac{q}{q-1}(2^e+1)> q^f,$$ we obtain $n\le 2^e,$ which means that $D_q(n)\le 2^e<q^f,$ contradiction. Conversely, 
$D_q(q^{f-1}(q-1))\le q^f$
by Lemma \ref{prop5n}. Since $ q^{f-1}(q-1)\ge 2^e+1,$ we infer that neither $2^e$ nor $2^{e+1}>q^f$ can be a discriminator for $q^{f-1}(q-1)$ and hence $q^f\in\cal D_q.$ 
\end{proof}

The following proposition gives a 
reformulation of the inequality
\eqref{ciolan} which is computationally very
easy to work with.
\begin{Prop}
\label{allowedexponent}
Let $q\ge 5$ be a prime. Put
$${\mathcal F}_q=\Big\{b\ge 1:\Big\{b\frac{\log q}{\log 2}\Big\}>\frac{\log(q/(q-1))}{\log 2}\Big\}.$$
The set ${\mathcal F}_q$ is the set of integers $b\ge 1$ for which there is an integer
$e$ such that
\begin{equation}
\label{weetniet}
{q\over q-1}(2^e+1)\le q^b<2^{e+1}.
\end{equation}
\indent Alternatively, it is the set of integers 
$b\ge 1$ such that the interval  $[(q-1)q^{b-1},q^b]$ does not
contain a power of $2$.
\end{Prop}
\begin{proof} 
The inequality \eqref{weetniet} is equivalent with $2^eq/(q-1)<q^b<2^{e+1}$. By taking logarithms and
easy manipulations this is seen to be equivalent with
$$\frac{\log(q/(q-1))}{\log 2}<f\frac{\log q}{\log 2}-e<1.$$
This inequality can only be satisfied if we take
$e=\lfloor f\log q/\log 2\rfloor$. We are left with the inequality
$$\Big\{f\frac{\log q}{\log 2}\Big\}>\frac{\log(q/(q-1))}{\log 2},$$
which finishes the proof.
\par For the second assertion, we 
let  ${\mathcal G}$ 
be the set of exponents $k\ge 1$ such that $(q-1)q^{g-1}\leq 2^k \leq q^g$ for some 
integer $g\ge 1$.
We have to show that $\mathcal G$ is the complement of ${\mathcal F}_q$ in the natural integers.
Note that $g$ is in ${\mathcal G}$ iff $\log(q-1) + (g-1)\log q \leq k \log 2 \leq g\log q$, that
is, iff
$\log(q-1)/\log 2+(g-1) \alpha\leq k \leq g \alpha,$
where $\alpha=\log q/\log 2$. Since $k$ is an integer, we may replace $g\alpha$ by $\lfloor g\alpha\rfloor$ and
the condition becomes $k\in [\lfloor g\alpha\rfloor+\{g\alpha\}+\log(q-1)/\log 2-\alpha,\lfloor g\alpha\rfloor]$. Note that there can be only
one integer $k$ in this interval iff $\{g\alpha\}\le \log(q/(q-1))/\log 2$.
\end{proof}

\indent The reader might wonder how sparse the
set ${\mathcal F}_q$ is. The following
result gives an asymptotic answer.
\begin {Prop}
\label{iza}
As $x\to\infty,$ we have
$$\#\{b\in {\mathcal F}_q:b\le x\}\sim \frac{\log(2(q-1)/q)}{\log 2} x.$$  
 \end{Prop}
 \begin{proof}
 It is easy to see that $\log q/\log 2$ is
 irrational.
 Now it is a consequence of Weyl's criterion 
that, 
for a fixed $0<\beta<1$ and an irrational 
$\alpha,$ we have
$$\#\{g\le x:\{g\alpha\}> \beta\}\sim (1-\beta) x,~x\rightarrow \infty.$$
To conclude, apply this result 
with $\alpha=\log q/\log 2$
and $\beta=\log(q/(q-1))/\log 2$.
\end{proof}
\noindent {\tt Remark}. Note that the proportionality constant in
Proposition \ref{iza} satisfies 
$$ \frac{\log(2(q-1)/q)}{\log 2}=1-\frac{1}{q\log 2}+O\left(\frac{1}{q^2}\right),~q\rightarrow 
\infty.$$

\noindent {\tt Commentary}. Lemma \ref{valuenotFermat} is new.
In Lemma \ref{prop5n} it is crucial to have
$\nu_q(3^{q-1}-1)=1$, that is, we need to have that $q$ is not Mirimanoff.
The proof also hinges on $u_q(1),\ldots,u_q(q-1)$ being pairwise distinct modulo
$q$, which happens iff $q$ is Artin. With these assumptions
on $q,$ the earlier proof for $q=5$ generalizes. 
Proposition \ref{iza} is a very straightforward 
generalization of \cite[Proposition 1]{MZ}.

\subsection{$D_q(n)$ is a power of $2$ or $q$}
\label{primesets}
Put 
$${\cal P}=\{p>3:p\ne q\text{~and~}\ord_p(9)=(p-1)/2\}.$$ 
Let 
$${\cal P}_j=\{p>3:p\ne q,~p\equiv j\pmod*4\text{~and~}\ord_p(3)=p-1\},~j\in \{1,3\},$$ 
and $${\cal P}_2=\{p>3:p\ne q,~p\equiv 3\pmod*4\text{~and~}\ord_p(3)=(p-1)/2\}.$$
We have
$${\cal P}_1=\{5,17,29,53,89,101,113,137,149,173,197,233,257,269,281,293,\ldots\},$$
$${\cal P}_2=\{11,23,47,59,71,83,107,131,167,179,191,227,239,251,263,\ldots\},$$
$${\cal P}_3=\{7,19,31,43,79,127,139,163,199,211,223,283,
\ldots\},$$
where, for any fixed $q$, if any of the primes listed equals $q,$ it has
to be removed from the corresponding set.
By (\ref{lcmorder}) we have $2\ord_p(9)=\lcm(2,\ord_p(3))$, from
where we infer that ${\cal P}={\cal P}_1 \cup {\cal P}_2 \cup {\cal P}_3$.
If a prime $p>3$ is a Browkin-S\u al\u ajan value, then by Lemma \ref{redpp} we must have $p\in {\cal P}$.
If $p\in {\cal P}$, then by Theorem \ref{generalperiod} we have $\rho(p)=p-1$. This will be used a few times in the
sequel.\\
\indent The aim of this section is to establish
the following result, the proof of which makes use of properties of the incongruence index and is
given in Section \ref{5.4.1}.
\begin{Prop}
\label{reducetoprime}
Let $d>1$ be an integer coprime to $2q$. If $d$ is a Browkin-S\u al\u ajan value, then 
$d\in {\cal P}$.
\end{Prop}
\subsubsection{The incongruence index}
\begin{Def}
Let $\{v_j\}_{j=1}^{\infty}$ be a sequence of integers and $m$ an integer. Then $\iota_v(m)$, the incongruence index  of $v$ modulo $m,$ is
the largest number $k$ such that $v_1,\ldots,v_{k}$ are pairwise incongruent modulo $m$.
\end{Def}
Note that $\iota_v(m)\le m$. In case the sequence $v$ is purely periodic
modulo $d$, we have $\iota_v(d)\le \rho_v(d)$. A minor change in the proof
of Lemma \ref{gee} yields the following result.
\begin{Lem}
\label{gee2}
Assume that $D_v(n)\le g(n)$ for every $n\ge 1$ with $g$ non-decreasing.
If $d>g(\iota_v(d))$, then $d$ is a $D_v$-non-value.
\end{Lem}
\par For the Browkin-S\u al\u ajan sequence $u_q$ we write $\iota_q(d)$ to highlight the dependence on $q.$ However, whenever the dependence on $q$ does not play a role, we will write $\iota(d)$ for simplicity. 
\par A minor variation of the proof of Lemma \ref{nonnie} gives the following result, which will be of vital importance in order to discard possible Browkin-S\u al\u ajan values.
\begin{Lem}
\label{nonnie2}
If $\iota(d)\le d/2$, then $d$ is a Browkin-S\u al\u ajan non-value.
\end{Lem}

\subsubsection{The incongruence index for $q^2$}
\label{qkwadraat}
In this section we consider the incongruence index 
for $q^2$. The result and its
corollary are not used in the sequel, but
shed some light on the behaviour of 
the incongruence index.
\par We define
$${\cal Q}=\{q>5:q\equiv 3\pmod*4\text{~and~}\ord_q(3)=(q-1)/2\},$$  and write 
$$\alpha_q=\frac{3^{(q-1)/2}-1}{q}.$$
Note that, if $q\in {\cal Q}$, then $\alpha_q$ is an integer and $q^*=-q$.
\begin{Lem}
\label{q^2}
Let $q\in {\cal Q}$ and 
suppose that $3^{(q-1)/2}\not\equiv 1\pmod*{q^2}$.\\
{\rm a)} If $2\alpha_q$ is a quadratic residue modulo $q,$ then
there exists a smallest integer $1\le m\le (q-1)/2$ such that 
$9^m\equiv 2/\alpha_q\pmod*q$ has a solution.
We have $u_q(2m)\equiv u_q(2m+(q-1)/2)\pmod*{q^2}$ and
$\iota_q(q^2)=2m-1+(q-1)/2$.\\
{\rm b)} If $2\alpha_q$ is a quadratic non-residue modulo $q,$ then
there exists a smallest integer $1\le m\le (q-1)/2$ such that 
$9^m\equiv -6/\alpha_q\pmod*q$ has a solution.
We have $u_q(2m-1)\equiv u_q(2m-1+(q-1)/2)\pmod*{q^2}$ 
and
$\iota_q(q^2)=2m-2+(q-1)/2$.
\end{Lem}
\begin{cor}
If $q\in {\cal Q}$ and $3^{(q-1)/2}\not\equiv 1\pmod*{q^2}$, 
then $\iota_q(q^2)\le 3(q-1)/2-1<q^2/2$.
\end{cor}
\begin{proof}[Proof of Lemma \ref{q^2}] Our argument uses that $q\nmid \alpha_q,$ which is a consequence of the assumption that 
$3^{(q-1)/2}\not\equiv 1\pmod*{q^2}$.\\
\indent For part a) we have to show that
$3^{2m}+q\equiv 3^{2m+(q-1)/2}-q\pmod*{4q^2}$. Since the congruence
clearly holds modulo $4$, it is enough to show that it holds modulo $q^2$; 
in other words, it is enough to show that
$2q\equiv 9^m\alpha_q q\pmod*{q^2}$. That this holds is a consequence
of the identity $2\equiv 9^m\alpha_q\pmod*q$. Our assumption on $q$ implies
that $\ord_q(9)=(q-1)/2$. Thus the subgroup of $(\mathbb Z/q\mathbb Z)^*$ generated
by $9$ is the subgroup of all squares. Since by assumption $2\alpha_q$ is a 
quadratic residue modulo $q$, so is $2/\alpha_q$ and hence there is a smallest
integer $1\le m\le (q-1)/2$ such that  $2\equiv 9^m\alpha_q\pmod*q$. 
We thus conclude that $\iota_q(q^2)\le 2m-1+(q-1)/2$. In order to establish
equality we notice that, if $r$ 
is the smallest number such that
$u_q(k)\equiv u_q(r)\pmod*q$ for 
some $1\le k<r$, then for general $q$ we have
$k\equiv r\pmod*{\ord_q(3)}$, and thus for our choice of $q$ we must
have $r\in \{k+(q-1)/2,k+(q-1),\ldots\}$. Since 
$r-1=\iota_q(q^2)\le 2m-1+(q-1)/2\le 3(q-1)/2-1$, we infer that $r=k+q-1$ or $r=k+(q-1)/2$. Two cases must be dealt with.\\
{\sc{Case 1.}} $r=k+q-1$.\\
Here $r$ and $k+q-1$ are of the same parity and we
must have $3^{q-1}\equiv 1\pmod*{q^2}$. In particular,
$u_q(1)\equiv u_q(q)\pmod*{q^2}$ and 
so $\iota_q(q^2)\le q-1$.\\ 
{\sc{Case 2.}} $r=k+(q-1)/2$.\\
Here $k$ and $r=k+(q-1)/2$ are of different parity.
If $k$ is odd, we can write it as $2v-1$ and then infer that 
$9^v\equiv -6/\alpha_q\pmod*q$, which has no solution as
$$\bigg(\frac{-6\alpha_q}{q}\bigg)=\bigg(\frac{-3}{q}\bigg)
\bigg(\frac{2\alpha_q}{q}\bigg)=-1.$$ Thus we
must have $k=2v$. We conclude that $\iota_q(q^2)=2v-1+(q-1)/2$, where 
$v$ is the smallest positive
integer such that $u_q(2v)\equiv u_q(2v+(q-1)/2)\pmod*{q^2}$, which 
yields $v=m$ as we have seen above.\\
\indent The proof of part b) is very similar and left to the interested reader.
\end{proof}

\subsubsection{Lifting from $p^m$ to $p^{m+1}$}
\begin{Lem}
\label{liftje}
Let $p>3$. Let $1\le t<m$. Then
\begin{equation}
\label{rhopm}
\rho(p^m)\,|\,\lcm(2,\rho(p^t))p^{m-t}
\end{equation}
and
\begin{enumerate}[\rm a)]
\item if $p^t\neq q$, then $\rho(p^m)\,|\,\rho(p^t)p^{m-t};$
\item if $p^m\neq q$, then either $\rho(p^{m+1})=\rho(p^m)$ or $\rho(p^{m+1})=p\rho(p^m);$
\item if $\rho(p^2)=p\rho(p)$, then $\rho(p^m)=p^{m-1}\rho(p)$ for $m\ge 2$.
\end{enumerate}

\end{Lem} 
\begin{proof} 
For notational convenience write  $\rho_1(p)=\lcm(2,\rho(p^t))$.
Since $u_q(k)\equiv u_q(k+\rho_1(p^t))\pmod*{p^t}$ for every $k\ge 1$, 
it follows that $3^{\rho_1(p^t)}\equiv 1\pmod*{p^t}$ and 
from this 
we obtain $3^{\rho_1(p^t)p^{m-t}}\equiv 1\pmod*{p^m}$ and hence we
deduce that $$u_q(k)\equiv u_q(k+\rho_1(p^t)p^{m-t})\pmod*{p^m}$$ for every $k\ge 1$ and 
so $\rho(p^m)\mid \rho_1(p^t)p^{m-t}$.\\
\indent Assertion a) follows from (\ref{rhopm}) since
$\rho(p^t)$ is even for $p^t\ne q$  by Theorem \ref{generalperiod}. 
Assertion b) follows from assertion a)
and the observation that $\rho(p^m)\mid \rho(p^{m+1})$. 
Finally, assertion c) is a consequence of Theorem \ref{generalperiod}
and Lemma \ref{Beyl}. \end{proof}
\begin{Lem} 
\label{indelift}
If $p>3$ and $\iota(p^m)<\rho(p^m)$, then $\iota(p^{m+1})<p^{m+1}/2$.
\end{Lem}
\begin{proof} 
Since $\iota(q)=\rho(q)$ we may assume that $p^m\ne q$. By Theorem \ref{generalperiod} this
implies that $\rho(p^m)$ is even. It then follows by part b) of Lemma \ref{liftje} that
either $\rho(p^{m+1})=\rho(p^m)$ or
$\rho(p^{m+1})=p\rho(p^m)$. In the first case
$$\iota(p^{m+1})\le \rho(p^{m+1})=\rho(p^m)\le p^m<p^{m+1}/2,$$ so we may
assume that $\rho(p^{m+1})=p\rho(p^m)$. This implies
that
\begin{equation}
\label{nonmirimanoff}
3^{\rho(p^m)}\equiv 1+kp^m\pmod*{p^{m+1}}
\end{equation}
with $p\nmid k$. From this we infer that
$u_q(i+j\rho(p^m))$ assumes $p$ different values modulo $p^{m+1}$ as
$j$ runs through $0,1,\ldots,p-1$. 
Put $j_1=\iota(p^{m})+1$. By assumption there
exists $1\le i_1<j_1<\rho(p^m)$ such that $u_q(i_1)\equiv u_q(j_1)\pmod*{p^m}$.
Modulo $p^{m+1}$ we have
$$\{u_q(i_1+j\rho(p^m)):0\le j\le p-1\}=\{u_q(j_1+j\rho(p^m)):0\le j\le p-1\}.$$ 
The cardinality of these sets is $p$. Now let us consider the subsets obtained from the
above two sets if we restrict $j$ to be $\le p/2$. Each contains
$(p+1)/2$ different elements. It follows that these sets must have an element
in common. Say we have
$$u_q(i_1+k_1\rho(p^m))\equiv u_q(j_1+k_2\rho(p^m))\pmod*{p^{m+1}},~0\le k_1,k_2\le p/2.$$
Since by assumption $i_1\not\equiv j_1\pmod*{\rho(p^m)}$, we have that
$$i_1+k_1 \rho(p^m)\ne j_1+k_2\rho(p^m).$$
The proof is completed on noting that 
$i_1+k_1 \rho(p^m)$ and $j_1+k_2\rho(p^m)$ are bounded above by
$$\iota(p^m)+1+(p-1){\rho(p^m)\over 2}\le (p+1){\rho(p^m)\over 2}\le (p+1){\varphi(p^m)\over 2}=
p^{m-1}{(p^2-1)\over 2}<{p^{m+1}\over 2},$$
where we used 
that, by assumption, $\iota(p^m)+1\le \rho(p^m)$ and Lemma \ref{previous}. \end{proof}

\begin{Lem}
\label{Bb}
Let $p>3$ and $k\ge 1$ an integer.
If $\iota(p^k)\le p^k/2$, then
 $\iota(p^m)\le p^m/2$ for every $m>k$.
\end{Lem}
\begin{proof} It suffices to prove the result for $m=k+1$ and then
apply induction. 
Note that by Lemma \ref{previous} we have
$\rho(p^{k+1})|\varphi(p^{k+1})$.\\
{\sc{Case 1.}} $\rho(p^{k+1})\le \varphi(p^{k+1})/2$.\\
It follows that  $\iota(p^{k+1})\le \rho(p^{k+1})\le \varphi(p^{k+1})/2\le p^{k+1}/2$.\\
{\sc{Case 2.}}  $\rho(p^{k+1})=\varphi(p^{k+1})$.\\
If $p^k=q$, then $\iota(p)=\rho(p)$ and so $k\ge 2$ and
by Theorem \ref{generalperiod} 
it follows that 
$\rho(p^k)=\varphi(p^{k})$. If $p^k\ne q$, then it also
follows by Theorem \ref{generalperiod}  that $\rho(p^k)=\varphi(p^{k})$.
Since $p>3$ we have
$$\iota(p^k)\le p^k/2<p^{k-1}(p-1)=\varphi(p^k)=\rho(p^k).$$
On applying Lemma \ref{indelift} we infer that also in this case 
$\iota(p^{k+1})\le p^{k+1}/2$. \end{proof}

\noindent On combining the latter two lemmas with Lemma \ref{nonnie2} we arrive at the following 
more appealing result.
\begin{Lem}\label{18} Let $p>3$.\\
\noindent {\rm a)} If 
$\iota(p)<\rho(p)$, then $p^2,p^3,\ldots$ are all 
Browkin-S\u al\u ajan non-values.\\
{\rm b)} If $\iota(p)\le p/2$, then $p,p^2,p^3,\ldots$ are all Browkin-S\u al\u ajan non-values.
\end{Lem}
\begin{proof} \medskip\hfil\break
a) If the conditions on $p$ are satisfied, then 
by Lemma \ref{indelift} it follows that $\iota(p^2)\le p^2/2$, 
which by Lemma \ref{Bb} implies that $\iota(p^m)\le p^m/2$ for every $m\ge 2$. 
By Lemma \ref{nonnie2} it then follows that $p^m$ is a non-value.\\
b) If $\iota(p)\le p/2$, then $\iota(p^m)\le p^m/2$ for every $m\ge 1$ by Lemma \ref{Bb} and by 
Lemma \ref{nonnie2} it then follows that $p^m$ is a non-value. \end{proof}


\noindent {\tt Commentary}. 
Section \ref{qkwadraat} is new.
Lemma \ref{liftje} is 
a generalization of the 
trivial \cite[Lemma 8]{MZ}, whereas Lemmas \ref{indelift} and \ref{18} are proved in a similar way as Lemmas 16, respectively 18 
from \cite{MZ}. 
\par The set ${\cal P}$ in \cite{MZ} was partitioned in three subsets, 
${\cal P}_1,~{\cal P}_2$ and ${\cal P}_3$. In \cite[Lemma 19]{MZ} it is
shown that, if $p$ is in ${\cal P}_3$, then 
$\iota(p)\le p/2$ and hence $p\not\in {\cal D}_5$. The argument given there cannot be
generalized to arbitrary $q$. However, we will see that the weaker statement that $\iota(p)<\rho(p)$ is true, which is enough for our purposes and shows that $p^2,p^3,\ldots$ cannot be Browkin-S\u al\u ajan values.\\
\indent Lemma \ref{Bb} is patterned after
\cite[Lemma 17]{MZ}, but a somewhat more elegant proof
is given now. In the earlier proof one should read
$p^{m-1}(1-1/p)$ instead of $p^{m-2}(1-1/p)$.

\subsection{If $p\in\cal P$, then $\iota(p)<\rho(p)$}
\label{tip}
Lemma \ref{nonnie2} in combination with the following lemma shows that every $p\in {\cal P}$ is a Browkin-S\u al\u ajan  non-value. Recall that, if $p\in {\cal P}$, then $\rho(p)=p-1$ and
that by definition $q\not\in \cal P$.
\begin{Lem}
\label{remainder}
If $ p\in  {\cal P}$, then $\iota(p)<\rho(p).$  
\end{Lem}
\begin{proof}
We will find 
solutions to the congruence $3^{2i-1}+q^*\equiv 3^{2j}-q^*\pmod*{4p}$ with $ 1\le i,j\le (p-1)/2$, which then gives $u_q(2i-1)\equiv u_q(2j)\pmod*{p}$ 
and yields that $\iota(p)<\max\{2i-1,2j\}\le p-1=\rho(p)$. The indices are here of 
different parity as focusing on terms with indices having the
same parity will give only 
$\iota(p)\le \rho(p)$. As trivially
$3^{2i-1}+q^*\equiv 3^{2j}-q^*\pmod*{4}$, it
is enough to consider the congruences only modulo
$p$. We will make use of the fact that
$\{3^{2k}\pmod*p:1\le k\le (p-1)/2\}$ 
swipes out all non-zero squares modulo $p$ 
and that 
the set $\{3^{2k-1}\pmod*p:1\le k\le (p-1)/2\}$ 
swipes out all non-squares modulo $p$ in
case $\big({3\over p}\big)=-1$. This is 
a consequence of our 
assumption that $\ord_p(9)=(p-1)/2$.
\\
\noindent{\sc{Case 1.}} $p\in\cal P_1\cup\cal P_3.$ \\
\noindent Note that $\big({3\over p}\big)=-1$.
If $q^*$ is not a quadratic residue mod $p$, then $q^*\equiv3^{2i-1}\pmod*{p}$ for some $1\le i\le (p-1)/2,$ therefore $$ 3^{2i}-3^{2i-1}=2\cdot3^{2i-1}\equiv 2q^*\pmod*{p},$$ which yields $u_q(2i)\equiv u_q(2i-1)\pmod*p.$
 If $q^*$ is a quadratic residue mod $p,$ then
 we have $q^*\equiv3^{2k} \pmod*{p},$ for some $1\le k\le (p-1)/2$ and we distinguish two subcases:\\
\noindent a) $-q^*$ is a quadratic non-residue mod $p.$ Then $-q^*\equiv 3^{2\ell-1}\pmod*p,$ for some $1\le \ell\le (p-1)/2,$ and we get $u_q(2\ell-1)\equiv u_q(2k)\equiv 0\pmod*{p}.$\\ \noindent 
b) $-q^*$ is a quadratic residue mod $p.$ Then $-q^*\equiv 3^{2h}\pmod*p,$ for some $1\le h\le(p-1)/2.$ If $h<(p-1)/2,$ then $u_q(2h+1)\equiv -3q^*+q^*=-q^*-q^*\equiv u_q(2h)\pmod*p.$ If $h=(p-1)/2,$ then $-q^*\equiv 1\pmod*p$ and $u_q(1)\equiv u_q(p-1)\equiv 2\pmod*p.$ \\\noindent 
{\sc{Case 2.}} $p\in\cal P_2.$\\
\noindent We have 
\begin{equation}
\label{zondag1}
 u_q(2m-1)\equiv u_q(2m) \pmod*p\Leftrightarrow 3^{2m}\equiv 3q^*\pmod*p ,
\end{equation}
 and  
 \begin{equation}
 \label{zondag2}
 u_q(2m)\equiv u_q(2m+1) \pmod*p\Leftrightarrow 3^{2m}\equiv -q^*\pmod*p.
 \end{equation} 
 Since $\big({3q^*\over p}\big)=\big({-3\over p}\big)\big({-q^*\over p}\big)=-\big({-q^*\over p}\big),$  it follows that either 
 $3q^*$ or $-q^*$ is a square modulo $p.$ 
 Thus $3^{2k}\equiv -q^*\pmod*p$ or $3^{2k}\equiv 3q^*\pmod*p$
 holds for some $1\le k\le (p-1)/2$. Note 
 that if $q^*\not\equiv -1,1/3 \pmod*p$, then
 either \eqref{zondag1} or \eqref{zondag2} is satisfied
 with $m=k<(p-1)/2$. Otherwise, we have 
 to deal with the following two subcases:\\
 \noindent a) $q^*\equiv 1/3 \pmod*p. $ Then 
$u_q(p-2)\equiv (1/3+ 1/3)/4=(1-1/3)/4\equiv u_q(p-1)\pmod*{p}$.\\
\noindent b) $q^*\equiv -1\pmod*p. $ Then 
 $u_q(1)=(3-1)/4= (1-(-1))/4\equiv u_q(p-1)\pmod* p.$
\end{proof}

\subsubsection{A long overdue proof}
\label{5.4.1}
Finally we have developed enough tools to live up to our
promise made at the 
end of Section \ref{primesets} and prove Proposition \ref{reducetoprime}. 
\begin{proof}[Proof of Proposition \ref{reducetoprime}]
Suppose that $(d,2q)=1$. By Lemma \ref{redpp} it follows that $d=p^m$ with $p>3$ and $p\in {\cal P}$ (hence $p\ne q$).
It follows from Lemma~\ref{remainder} that $\iota(p)<\rho(p)$ for 
every $p\in \mathcal P,$ which implies by Lemma \ref{18} that $m=1$ and $d=p$.
\end{proof}

\noindent{\tt{Commentary.}} The first case of the proof of Lemma \ref{remainder} is the counterpart of \cite[Lemma 19]{MZ}, while the second is that of \cite[Lemma 20]{MZ}.

\subsection{$D_q(n)$ is not a `big' prime}  
We will now use classical exponential sum techniques to show that, for sufficiently large primes, the condition given in Corollary~\ref{prop2n} is not satisfied. Therefore, big primes are Browkin-S\u al\u ajan non-values.

Let us denote by $\psi$ the additive characters of the group $G$ and $\psi_0$ the trivial character. For any non-empty subset $A\subseteq G$, let us define the quantity \begin{equation}\label{def_S(A)}
|\widehat{A}|=\max_{\psi\neq \psi_0} \left| \sum_{a\in A} \psi(a)\right|,
\end{equation}
where the maximum is taken over all non-trivial characters in $G$.

The following result is Lemma 21 in \cite{MZ}.
\begin{Lem}\label{B+B_zero} Let $G$ be a finite abelian group. For any given non-empty subsets $A,B\subseteq G$, whenever $A\cap (B+B)=\emptyset$ we have 
$$|B|\le {|\widehat{A}||G|\over |A|+|\widehat{A}|},$$ where $|\widehat{A}|$ is the quantity defined in~\eqref{def_S(A)}.
\end{Lem}

We will need the following auxiliary result, which can be found in~
Cilleruelo and Zumalac\'arregui~\cite{Ana}. 
\begin{Lem}
\label{S(A)} 
Let $g$ be a primitive root modulo $p$ and  $a$, $b$ and $c$ be integers 
such that $p\nmid abc$. Then the set 
\[
A_g(p;a,b,c)=\{ (x,y):\, ag^x-bg^{y} \equiv c \pmod* p\}\subset \mathbb Z_{p-1} \times \mathbb Z_{p-1}
\]
has $p-2$ elements and satisfies $|\widehat{A}_g(p;a,b,c)|<\sqrt{p}$.
\end{Lem}

\begin{Prop}\label{prime_lemma} Let $p>3$ be a prime with $p\ne q$. 
Suppose that $u_q(1),\ldots,u_q(n)$ are pairwise distinct modulo $p$. Then $p> \left\lfloor \frac{n}{4} \right\rfloor^{4/3}$.
\end{Prop}
\noindent \begin{proof} First observe that, if two elements have the same
parity index, then
$u_q(i)\not\equiv u_q(i+2k) \pmod*p$ iff $
9^k\not\equiv 1\pmod*p,$ thus $\ord_p(9)\ge n/2$. 
(Alternatively one might invoke Lemma \ref{oneventje} to obtain this conclusion.)
By hypothesis, on
comparing elements with distinct parity index, it follows that
\begin{equation}
\label{power3} 3\cdot 9^{k}-9^{s} \equiv 6q^*\pmod* p,\ 1\leq
k,s\leq \left\lfloor \tfrac{n}{2}\right\rfloor
\end{equation}
has no solution (otherwise $u_q(2k)\equiv u_q(2s-1) \pmod* p$, with $1\le 2k,
2s-1\le n$).

We will now show that the non-existence of solutions to 
equation~\eqref{power3} implies that $p >\lfloor \frac{n}{4} \rfloor^{4/3}$.
Let $g$ be a primitive root modulo $p$ and let $A_g(p;3,1,6q^*)$ be the set defined in
Lemma~\ref{S(A)}. Let $m$ be the smallest integer such that
$g^m\equiv 9\pmod* p$ and put
$$B=\{(mx,my):1\le x,y\le  \lfloor n/4\rfloor\} \subset \mathbb Z_{p-1} \times \mathbb Z_{p-1}.$$
Note that, since $\ord_p(9)\ge n/2$, it follows that $|B|=\left\lfloor
\frac{n}{4} \right\rfloor^2$ (since $m$ generates a subgroup of order at
least $n/2$ modulo $p-1$).

Observe that the non-existence of solutions to
equation~\eqref{power3} implies that
\begin{equation*}
\label{power4} 3\cdot g^{mk}-g^{ms} \equiv 6q^* \pmod* p,\ 1\leq
k,s\leq \left\lfloor \tfrac{n}{2}\right\rfloor
\end{equation*}
has no solutions and in particular $A_g(3,1,6q^*)\cap (B+B)=\emptyset$ (since clearly
$B+B\subseteq \{(mx,my): 1\le x,y\le \left\lfloor n/2\right\rfloor \}$). It
follows from Lemma~\ref{B+B_zero} and Lemma~\ref{S(A)} that
\begin{equation}
\label{bijna}
|B|=\left\lfloor \frac{n}{4} \right\rfloor^2 \le
\frac{|\widehat{A}||G|}{|A|+|\widehat{A}|}\le
\frac{p^{1/2}(p-1)^2}{p-2+p^{1/2}}<
p^{3/2},
\end{equation}
which concludes the proof.
\end{proof}

\begin{cor}\label{Nobigprimes}
If $p > 2060$ is a prime number with 
$p\ne q$, then $p$ is a Browkin-S\u al\u ajan non-value.
\end{cor}
\begin{proof} First observe that, if $n\ge 2060,$ then it follows from Proposition~\ref{prime_lemma}  that if, for some prime $p\ge n,$ the 
elements $u_q(1),\ldots,u_q(n)$ are pairwise distinct modulo $p$, then
\[
p> \left\lfloor \frac{n}{4} \right\rfloor^{4/3} \ge 2n,
\]
and by Lemma~\ref{prop2n} it follows that $p$ is a 
Browkin-S\u al\u ajan non-value.\end{proof}


\noindent {\tt Commentary}. 
In the proof of 
Proposition \ref{prime_lemma} we now need to consider the more general
sets $A_g(p;3,1,6q^*)$ instead of the sets $A_g(p;3,1,30)$. 
As these behave in the same way as $A_g(p;3,1,30)$, provided
$p\ne q$, the proof is very similar to that of the corresponding
Proposition 4 in \cite{MZ}.

\subsection{Primes $p <2060$ that can occur}
A final step in \cite{MZ} was to check numerically that no prime $5<p<2060$ can occur as discriminator for the S\u al\u ajan sequence $u_5.$ For our more general Browkin-S\u al\u ajan sequence $u_q,$ this is no longer true. Numerical computations reveal, for instance, that $ D_q(5)=7 $ for certain values of $q$. By computer verification we will see, in fact, that 7 is the only such exceptional value, and Lemmas \ref{7remains!} and \ref{Dq(5)} will clarify when it occurs. 
\begin{Def}
Given a prime $p$, we define the universal
incongruence index as
$$\upsilon(p)=\max\{\iota_q(p):q\ne p,~q\ge 5\},$$
where $q$ ranges over the primes $q\ge 5$.
\end{Def}
The following easy property of the incongruence
index allows one to compute $\upsilon(p)$.
\begin{Lem}
Let $5\le q_1<q_2$ be two primes such 
that $q_2\equiv \pm q_1\pmod*{4p}$, then we have
$\iota_{q_1}(p)=\iota_{q_2}(p)$.
\end{Lem}
\begin{proof}
Follows on noting that $u_{q_1}(n)\equiv u_{q_2}(n)
\pmod*{p}$ for every $n\ge 1$.
\end{proof}
\begin{Lem}
\label{upsilon}
We define $$S(p;r)=\{3\cdot 9^x-9^y\pmod*p:1\le 
2x,~2y-1\le r\}\cup\{0\}.$$
If $p\in \cal{P}$, then 
$\upsilon(p)=h(p)$, where
$$h(p)=\max \{r:S(p;r)\neq \mathbb Z/p\mathbb Z\}$$ is well-defined. 
\end{Lem}
\begin{proof}
An equal parity argument only yields that 
$\upsilon(p)\le \rho(p)$.
By Lemma \ref{remainder} the assumption 
$p\in \cal P$ implies that $\iota(p)<\rho(p)$.
Thus the smallest $\ell$ for which there exists $1\le k<\ell$
and
\begin{equation}
\label{eqi}
u_q(k)\equiv u_q(\ell)\pmod*{p}
\end{equation}
has a distinct parity
from $k$.
\par Thus we obtain a congruence of
the form $u_q(2x)\equiv u_q(2y-1)\pmod*{p}$,
which is equivalent with
\begin{equation}
\label{geval}
3\cdot 9^x-9^y\equiv 6q^*\pmod*{p}.
\end{equation}
\par First suppose that 
$S(p;r)\ne \mathbb Z/p\mathbb Z$.
If $a\not\in S(p;r)$, then
for those $q$ satisfying 
$q^*\not\equiv a/6\pmod*{p},$ we have that 
\eqref{geval} is not satisfied with 
$1\le 2x,2y-1\le r$ and so
$\iota_q(p)\ge r$. 
By Dirichlet's theorem for primes in arithmetic progression
there are indeed primes $q$  satisfying 
$q^*\not\equiv a/6\pmod*{p}$. It follows that
$\upsilon(p)\ge r$. For $r>\rho(p)$ we have
$S(p;r)=\mathbb Z/p\mathbb Z$. Note 
that if $S(p;r_0)=\mathbb Z/p\mathbb Z$ for 
some $r_0$, then $S(p;r)=\mathbb Z/p\mathbb Z$ for
every $r>r_0$. We thus conclude that $h(p)$
is well-defined and that $\upsilon(p)\ge h(p)$.
\par Next suppose that $S(p;r)= \mathbb Z/p\mathbb Z$.
Then, whatever $q\ne p$ we choose, the congruence
\eqref{geval} has a solution with $1\le 2x,2y-1\le r$.
We conclude that $\iota_q(p)\le \upsilon(p)<r$
and $\upsilon(p)<h(p)+1$. This inequality, together
with $\upsilon(p)\ge h(p)$ finishes the proof.
\end{proof}
\begin{Lem}
\label{throwprimesout}
Let $p\in \cal P$.
If there is a power of $2$ in the interval 
$[h(p),p)$, then
$p$ is a Browkin-S\u al\u ajan non-value.
\end{Lem}
\begin{proof}
By contradiction. Recall that $p\in \cal P$ implies
that $p\ne q$.
If $D_q(n)=p$ for some $n$, then $n\le \iota_q(p)
\le \upsilon(p)=h(p)$
by Lemma \ref{upsilon}. Now if there is a power of two, say
$2^e$, in
the interval $[h(p),p)$, it discriminates the first
$h(p)$ values of $u_q$. As $2^e<p$, it follows that $D_q(n)\le 2^e$. Contradiction.
\end{proof}
\begin{cor}\label{criterionpower2}
Let $p\in \cal P$. If $h(p)\le (p+1)/2,$
then $p$ is a Browkin-S\u al\u ajan non-value.
\end{cor}
This corollary gives a very powerful and easy to implement criterion to exclude small values of $p$ from the possible Browkin-S\u al\u ajan values. By numerical work done in Maple and Mathematica, we infer that $h(p)\le (p+1)/2$ for all primes $31\le p<3000,$ $p\in\cal P,$ see Table \ref{valuesh}. Finally, by Lemma \ref{throwprimesout} we are left only with $p=7$ as potential exception.
\begin{table}[h]

\begin{tabular}{|c|c|c|c|}

\hline
$p$ & $h(p)$ & $p$ & $h(p)$ \\ 
\hline\hline
$5$ & $3$ & $31$ & $16$\\ 
\hline
$7$ & $5$ & $43$ & $21$\\ 
\hline
$11$ & $7$ & $47$ & $20$\\ 
\hline
$17$ & $11$ & $53$ & $20$\\ 
\hline
$19$ & $11$ & $59$ & $23$\\ 
\hline
$23$ & $12$ & $71$ & $25$\\ 
\hline
$29$  & $16$ & $79$ & $27$\\
\hline
 
\end{tabular}
\caption{Values of $h(p)$ for $p$ in $\cal P$}\label{valuesh}
\end{table}
\begin{Lem}
\label{7remains!}
Suppose that $D_q(n)=p$ with $p\ne q$ a prime.
Then $n=5$ and $p=7$.
\end{Lem}
\begin{proof} 
We note that 7 can only be a discriminator for $n=5$ and $n=6.$
Namely, 4 is a discriminator 
for $n\le 4$ and 7 discriminates at
most 6 values as $u_q(1)\equiv u_q(7)\pmod*{7}$.
It is not difficult to show that $D_q(6)=7$ if $q=7$ and
$D_q(5)=8$ in all other cases. Thus we conclude
that $n=5$.
\end{proof}

\noindent {\tt Commentary}.  In \cite{MZ} the idea was
to bound $\iota_5(p)$ by $(p-1)/2$, leading to
the conclusion that 
$p$ is a  Browkin-S\u al\u ajan non-value. Here the
basic idea is the same, but now with
$\upsilon(p)=\{\iota_q(p):q\ne p,~q\ge 5\}$. That
turns out to be rather more difficult and so this
section is mainly new.

\subsection{Discriminator values for 
small fixed $n$}
Obviously as $n$ is fixed and
$q$ ranges over the primes $\ge 5$, $D_q(n)$ can 
assume only finitely many possible values.
Indeed, trivially one has $D_q(1)=1$, $D_q(2)=2$, $D_q(3)=4$ and 
$D_q(4)=5$. The values $D_q(5)$ and $D_q(6)$ 
are slightly more difficult to determine.
\begin{Lem}
\label{Dq(5)}
We have
$$D_q(5)=
\begin{cases}
7 & \text{~if~}q=7\text{~or~}q\equiv \pm 1\pmod*{28};\\
8 & \text{~otherwise,}
\end{cases}
\text{~and~}D_q(6)=
\begin{cases}
7 & \text{~if~}q=7;\\
8 & \text{~otherwise.}
\end{cases}
$$
\end{Lem}
\begin{proof}
Writing, say, $q=4k+ 1,$ the Browkin-S\u al\u ajan sequence reads as $$k+1,2-k,k+7,20-k,61+k,182-k,547+k,\ldots.$$
Simply by testing all residue classes of $k$ modulo 7 
one concludes that $D_q(5)=7$ iff $k\equiv0\pmod* 7.$ If $q=4k+3,$ the sequence becomes 
$$-k,3+k,6-k,21+k,60-k,183+k,546-k,\ldots $$
and, by the same method, one concludes that $D_q(5)=7$ iff
$k\equiv 6\pmod*7$ or $q=7.$ Also, one sees that $D_q(6)=7$ 
iff $q=7.$ 
\end{proof}


\noindent {\tt Commentary}. There is no counterpart of this
in \cite{MZ}.

\section{The proof of the main result}\label{proofmainresult}
In Section~\ref{preparations}, we established that powers of $2$ and powers of prime numbers $p>3$ are candidates for Browkin-S\u al\u ajan values. On fixing the prime $q$, it is seen by Lemma \ref{prop5n} that 
powers of $q$ itself are candidates too. 
Finally, after studying the characteristics of the period and the incongruence index of the Browkin-S\u al\u ajan sequence, we discarded in Section~\ref{non-values} any other possible candidates, except for the value
$7$ for certain primes $q$.

\begin{proof}[Proof of Theorem \ref{Bmain1}]
It follows from Proposition~\ref{reducetoprime} that, if $d>1$ is a 
Browkin-S\u al\u ajan value, then either $(2q,d)>1$ or $d\in \mathcal P$. 
By Lemma \ref{7remains!}, if $D_q(n)=p$ for some
integer $n\ge 1$ 
and some odd 
prime $p\ne q$, then $n=5$ and $p=7$. It is easy
to check that the predicted value for $D_q(5)$ 
in the statement of the theorem matches the 
actual value given in Lemma \ref{Dq(5)}.
Thus from now on,
we may assume that 
$n\ne 5$ and $(2q,d)>1$. Then, by 
Lemma~\ref{redpp}, $d$ has to be a prime power 
and hence
the discriminator must be a power of $2$ or a power of $q$.

Note that $2^e$ discriminates $u_q(1),\ldots,u_q(n)$
iff $2^e\ge n$. Our analysis splits into several cases.\\
{\sc{Case 1.}} $q$ is Artin and not Mirimanoff.\\
By Lemma
\ref{prop5n} it follows that 
$q^f$ discriminates 
$u_q(1),\ldots,u_q(n)$
iff $q^f\ge qn/(q-1)$.
We infer that 
$$D_q(n)=\min\{2^e,q^f:2^e\ge n,~q^f\ge \frac{q}{q-1}n\}.
$$
\noindent {\sc{Case 2.}} $q$ is Artin and Mirimanoff.\\
By Lemma \ref{redpp} we must have $f=1$. In case $q$ is Fermat, by 
Lemma \ref{valuenotFermat} there is no $n$ with $D_q(n)=q$.
Note that $\iota(q)=q-1$ if $q$ is Artin and hence we must have
$q\ge n+1$. Suppose that $q$ is not Fermat and
$2^e$ is the largest power of $2$ less than $q-1$.
Then $D_q(n)=q$ for $n=2^e+1,\ldots,q-1$. Thus
we have showed that
$D_q(n)$ equals
$$
D_q(n)=
\begin{cases}
\min\{2^e,q:2^e\ge n,~q\ge n+1\} & \text{if~}q\text{~is Artin, Mirimanoff, but not Fermat};\\
\min\{2^e:2^e\ge n\} & \text{if~}q\text{~is Artin, Mirimanoff and Fermat}.
\end{cases}
$$
{\sc{Case 3.}} $q$ is not Artin.\\ 
It follows by Lemma \ref{qpowerspresent} 
that $q$ and its powers are all 
Browkin-S\u al\u ajan non-values.
Thus in this case we have 
$$D_q(n)=\min\{2^e:2^e\ge n\}.$$
\par This proves that the four part formula 
for $D_q(n),$ together with the exceptional case given in 
the statement of the theorem, 
is correct. 
\par As obviously all powers of $2$ occur, it remains
to determine which powers of $q$ do occur.
This we did in Lemma \ref{qpowerspresent}. On invoking 
Proposition \ref{allowedexponent},
the proof is completed.
 \end{proof}
 
\noindent {\tt Commentary}. This proof is considerably
more involved than in case $q=5$, as there are now eight
cases to be considered.

\section{Special primes}
\label{speciaal}
We recapitulate some material on Artin, Fermat and Mirimanoff primes.

\subsection{Artin primes}
Recall that an `Artin prime' we call a prime 
$q$ such that  $3$ is a primitive root modulo $q$.
How special are Artin primes? How
many Artin primes $q\le x$ are there? This is related to the celebrated Artin
primitive root conjecture. We refer to the appendix of \cite{MZ} for more information,
or Moree \cite{Msurvey} for much more information.

\subsection{Fermat primes}
A Fermat prime is a prime of the form $2^m+1$ with $m\ge 1$. It is a trivial observation
that we must have $m=2^e$. Currently the only Fermat primes known are
$3,5,17,257$ and $65537$.

\begin{Lem}
\label{3Fermat}
If a prime $q>3$ is Fermat, then $q$ is Artin.
\end{Lem}
\begin{proof} Note that it is enough to show that $\big({3\over q}\big)=-1$. Now apply
the law of quadratic reciprocity (details left to the reader). \end{proof}

\subsection{Mirimanoff primes}
Currently there are only two Mirimanoff primes known, namely $11$ and $1006003$,
see Keller and Richstein \cite{KR}. The prime $1006003$ is Artin, 
but $11$ is not. The Mirimanoff primes arose in the study of Fermat's Last 
Theorem, see, e.g., Ribenbom \cite{FLT} or Ribenboim \cite[Chapter 8]{FLT2}.
\subsection{Fermat-Mirimanoff primes}
A prime that is both Fermat and Mirimanoff we call a Fermat-Mirimanoff prime. 
Currently no such prime is known and perhaps they do not exist at all.
Note that by Lemma \ref{3Fermat} every Fermat-Mirimanoff prime is an
Artin-Fermat-Mirimanoff prime.

\section{Some numerical results}\label{numericalresults}
\subsection{Theorem \ref{Bmain1} in action}
In Tables \ref{tableq=5}--\ref{tableq=17} we demonstrate Theorem \ref{Bmain1}
in case $q=5,7,11,17,$ and $q=29.$ Highlighted are, in each case, the exceptional value 7 and the powers of $q.$ 

\begin{table}[h]
\begin{tabular}{|c|c|c|c|}
\hline
$n$ & $D_q(n)$ & $n$ & $D_q(n)$ \\
\hline\hline
$1$ & $1$ & $129-256$ & $256$\\
\hline
$2$ & $2$ & $257-512$ & $512$\\
\hline
$3-4$ & $4$ & $513-1024$ & $1024$\\
\hline
$5-8$ & $8$ & $1025-2048$ & $2048$\\
\hline
$9-16$ & $16$ & $2049-2500$ & $\underline{3125}$\\
\hline
$17-20$ & $\underline{25}$ & $2501-4096$ & $4096$\\
\hline
$21-32$ & $32$ & $4097-8192$ & $8192$\\
\hline
$33-64$ & $64$ & $8193-12500$ & $\underline{15625}$\\
\hline
$65-100$ & $\underline{125}$ & $12501-16384$ & $16384$\\
\hline
$101-128$ & $128$ & $16385-32768$ & $32768$\\
\hline
\end{tabular}
\caption{$q=5;$ $q$ is Artin, Fermat, but not Mirimanoff}\label{tableq=5}
\end{table}

\begin{table}[h]

\begin{tabular}{|c|c|c|c|}

\hline
$n$ & $D_q(n)$ & $n$ & $D_q(n)$ \\
\hline\hline
$1$ & $1$ & $129-256$ & $256$\\
\hline
$2$ & $2$ & $257-294$ & $\underline{343}$\\
\hline
$3-4$ & $4$ & $295-512$ & $512$\\
\hline
$5-6$ & $\underline{7}$ & $513-1024$ & $1024$\\
\hline
$7-8$ & $8$ & $1025-2048$ & $2048$\\
\hline
$10-16$ & $16$ & $2049-2058$ & $\underline{2401}$\\
\hline
$17-32$ & $32$ & $2059-4096$ & $4096$\\
\hline
$33-42$ & $\underline{49}$ & $4097-8192$ & $8192$\\
\hline
$43-64$ & $64$ & $8193-16384$ & $16384$\\
\hline
$65-128$ & $128$ & $16385-32768$ & $32768$\\
\hline

\end{tabular}
\caption{$q=7;$ $q$ is Artin, not Fermat and not Mirimanoff}\label{tableq=7}
\end{table}

\begin{table}

\begin{tabular}{|c|c|c|c|}

\hline
$n$ & $D_q(n)$ & $n$ & $D_q(n)$ \\
\hline\hline
$1$ & $1$ & $129-256$ & $256$ \\   
\hline
$2$ & $2$ & $257-512$ & $512$ \\    
\hline
$3-4$ & $4$ & $513-1024$ & $1024$\\   
\hline
$5-8$ & $8$ & $1025-2048$ & $2048$\\    
\hline
$9-16$ & $16$ & $2049-4096$ & $4096$\\  
\hline
$17-32$ & $32$ & $4097-8192$ & $8192$\\  
\hline
$33-64$ & $64$ & $8193-16384$ & $16384$\\
\hline
$65-128$ & $128$ & $16385-32768$ & $32768$\\
\hline

\end{tabular}
\caption{$q=11;$ $q$ is not Artin, not Fermat, but Mirimanoff}\label{tableq=11}
\end{table}


\begin{table}

\begin{tabular}{|c|c|c|c|}
\hline
$n$ & $D_q(n)$ & $n$ & $D_q(n)$ \\
\hline\hline
$1$ & $1$ &  $257-272$ & $\underline{289}$ \\   
\hline
$2$ & $2$ & $273-512$ & $512$\\  
\hline
$3-4$ & $4$ & $513-1024$ & $1024$\\  
\hline
$5-8$ & $8$ & $1025-2048$ & $2048$\\
\hline
$9-16$ & $16$ & $2049-4096$ & $4096$\\  
\hline
$17-32$ & $32$ & $4097-4624$ & $\underline{4913}$\\
\hline
$33-64$ & $64$ & $4625-8192$ & $8192$\\ 
\hline
$65-128$ & $128$ &  $8193-16384$ & $16384$\\  
\hline
$129-256$ & $256$ & $16385-32768$ & $32768$\\
\hline
\end{tabular}
\caption{$q=17;$ $q$ is Artin, Fermat, but not Mirimanoff}\label{tableq=17}
\end{table}

\begin{table}[t]

\begin{tabular}{|c|c|c|c|}

\hline
$n$ & $D_q(n)$ & $n$ & $D_q(n)$ \\
\hline\hline
$1$ & $1$ & $129-256$ & $256$\\
\hline
$2$ & $2$ & $257-512$ & $512$\\
\hline
$3-4$ & $4$ & $513-812$ & $\underline{841}$\\
\hline
$5$ & $\underline{7}$ & $813-1024$ & $1024$\\
\hline
$6-8$ & $8$ & $1025-2048$ & $2048$\\
\hline
$9-16$ & $16$ & $2049-4096$ & $4096$\\
\hline
$17-28$ & $\underline{29}$ & $4097-8192$ & $8192$\\
\hline
$29-32$ & $32$ & $8193-16384$ & $16384$\\
\hline
$33-64$ & $64$ & $16385-23548$ & $\underline{24389}$\\
\hline
$65-128$ & $128$ & $23549-32768$ & $32768$\\
\hline
\end{tabular}
\caption{$q=29;$ $q$ is Artin, not Fermat and not Mirimanoff}\label{tableq=29}
\end{table}

\subsection{Prime distribution over the eight possible cases in Theorem \ref{Bmain1}}
\label{primedistribution}
Theorem \ref{Bmain1} leads to eight possible cases 
if we take into account the exceptional case where
$D_q(5)=7$ and $q\ne 7$.
These
are listed in Table \ref{eightcases}. For each case we give the first few 
examples. In three cases there are  no known 
examples. Coming up with such an example would
require finding a Fermat prime
larger than 65537 or a Mirimanoff prime larger than 1006003.
Beyond examples, we give in Table \ref{eightcases} a conjectural natural
density of the primes belonging to each subcase. These
are all rational multiples of the Artin constant $A$ 
defined in \eqref{Artinconstant}.
\par We now explain how Table \ref{eightcases} has to be read. 
In the first column we indicate whether or
not the condition $q\equiv \pm 1\pmod*{28}$
is met. 
If an entry is empty in, say, the `Fermat' column, then this
means that both Fermat and non-Fermat primes are 
allowed. The final column lists the first few examples.
\par In Table \ref{densitiestable} we list
the conditional densities of the sets of primes
belonging to each of the eight cases.
We can only prove that these densities are true under
one or both of the following assumptions:\begin{itemize}
\item [(G)] The Generalized Riemann Hypothesis.
\item [(M)] The Mirimanoff primes have natural density zero.\end{itemize}
Which assumptions we make in order to establish the density
are indicated in the first column. The column 
`Empirical' rests on a Maple computation using the first million prime numbers.

\begin{table}[t]
\begin{tabular}{|c|c|c|c|c|}
\hline
 ${\pm 1\pmod*{28}}$ & Artin & Mirimanoff & Fermat &  Examples\\
\hline\hline
 yes & yes & no &   & $29,113,197,223,281,\ldots$\\
\hline
no & yes & no &  & $5,7,17,19,31,43,53,79,\ldots$\\
\hline
yes & yes & yes & no  & none known\\
\hline
no & yes & yes & no &  $1006003,\ldots$\\
\hline
yes & yes & yes & yes &  none known \\
\hline
no & yes & yes & yes &  none known \\
\hline
yes & no & &  & $83,167,251,307,337,\ldots$\\
\hline
no & no & & & $11,13,23,37,41,47,59,\ldots$\\
\hline
\end{tabular}
\caption{The eight prime sets arising in Theorem \ref{Bmain1}}\label{eightcases}
\end{table}

\begin{table}[b]
\begin{tabular}{|c|c|c|c|}
\hline
 Assumption & Density & Numerical & Empirical \\
\hline\hline
G, M & $32A/205$ & $0.05837\ldots$  & $\approx 0.0584$\\
\hline
G, M & $173A/205$ & $0.31558\ldots$ & $\approx 0.3155$\\
\hline
M &  $0$ & 0 &   \\
\hline
M &  $0$ & 0 & \\
\hline
M &  $0$ & 0 &  \\
\hline
M &  $0$ & 0 &  \\
\hline
G &  $1/6-32A/205$ & $0.10829\ldots$ & $\approx 0.1083$\\
\hline
G &  $5/6-173A/205$  & $0.51775\ldots$ & 
$\approx 0.5178$\\
\hline
\end{tabular}
\caption{Conjectural densities of 
the eight prime sets arising in Theorem \ref{Bmain1}}\label{densitiestable}
\end{table}

\par We determine the density
in the first case given in
Table \ref{densitiestable}. If one assumes 
G and M, then it is given by Lemma \ref{blubblub}. 
Using that, under GRH, the density of Artin primes is
$A$ (see, e.g., \cite[Theorem 1.2]{AP2}) and that, 
by Dirichlet's theorem on primes in arithmetic progressions, 
the density of the primes $q\equiv \pm 1\pmod*{28}$
is $1/6$, the remaining densities are easily obtained.
\begin{Lem}[GRH]
\label{blubblub}
The density of primes $q\equiv \pm 1\pmod*{28}$ 
that are Artin equals $32A/205$.
\end{Lem}
\begin{proof}[Sketch of proof]
Let $K$ be a number field.
Then the natural density of
primes $p$ that split completely
and have 3 as a primitive root exists and is given by
$$\sum_{n=1}^{\infty}\frac{\mu(n)}{[K(\zeta_n,3^{1/n}):\mathbb Q]}.$$ Using that
the primes $p$ that split completely in $\mathbb Q(\zeta_{28})$ are precisely the primes $p\equiv \pm 1\pmod*{28}$ we see that the density 
we are after equals
\begin{equation}
\label{maxreal}
\sum_{n=1}^{\infty}\frac{\mu(n)}{[\mathbb Q(\zeta_{28}+\zeta_{28}^{-1},\zeta_n,3^{1/n}):\mathbb Q]}.
\end{equation}
By some algebraic number theory 
making use of the fact that
$\mathbb Q(\zeta_{28}+\zeta_{28}^{-1})$
is the compositum of 
$\mathbb Q(\sqrt{7})$ and
the cubic real field 
$\mathbb Q(\zeta_{7}+\zeta_{7}^{-1}),$
we are led to the following degree evaluation in
case $4\nmid n$,
\begin{equation*}
\label{petergraad}
[\mathbb Q(\zeta_{28}+\zeta_{28}^{-1},\zeta_n,3^{1/n}):\mathbb Q]
=
\begin{cases}
n\varphi(n) & \text{~if~}42|n;\cr
2n\varphi(n) & \text{~if~}7|n\text{~and~}6\nmid n;\cr
\varphi(\lcm(28,n))n/2 & \text{~if~}7\nmid n.
\end{cases}
\end{equation*}
Note that since the
M\"obius function is zero for non-squarefree
numbers, it is enough to compute the degree in
case $4\nmid n$. After some calculations using the
Euler product in the form
$\sum_{(n,m)=1}\mu(n)f(n)=\prod_{p\nmid m}(1-f(p)),$
the proof is completed.
\end{proof}
The reader interested in working out the
details is referred to Moree \cite{28,AP2,Msurvey} for
similar computations that are worked out in
more detail.
 Alternatively, we have the following rigorous proof.
\begin{proof}[Second proof] By \cite[Theorem 1.2]{AP2}
we find that, under GRH, the density of the set of primes
$q\equiv 1\pmod*{28}$, respectively
$q\equiv -1\pmod*{28}$,
that are Artin, is $18A/205$, respectively 
$14A/205$.
\end{proof}
The above proof shows that the Artin primes are not
equidistributed over the
primitive residue classes modulo 28. Indeed, by Moree \cite[Theorem 1]{28}
they are not equidistributed over the
primitive residue classes modulo $d$ for any $d\ge 3$.
\par As a curiosity, we point out that the
set of primes $p$ such that 2 is a primitive
root modulo $p$ and $p$ is in various 
residue classes modulo $28$ appeared in
work of Rodier \cite{Rodier} in connection
with a coding theoretical problem involving
Dickson polynomials.\\

\noindent {\tt Commentary}.
This section is new. In \cite[Appendix A]{MZ} the
same method was used to deduce that, 
assuming the Generalized Riemann Hypothesis,
$\delta(\cal P_1)=\delta(\cal P_2)=3A/5$
and $\delta(\cal P_3)=2A/5$.\pagebreak



\par \noindent {\tt Acknowledgement.} In September 2015, Prof. Jerzy Browkin was invited by the second author to visit the Max Planck Institute for Mathematics, Bonn. The purpose was to
help guide some interns on discriminator
problems, on which Browkin himself also
published \cite{BC}. He gave a lecture 
on discriminators aimed at the interns and proposed problems.
The authors are grateful that Prof. Browkin, 
given his advanced age, was willing to make
the trip to Bonn and for sharing his ideas.
\par His passing away, a few months  after the visit, came as very sad and unexpected news for the authors. This paper is dedicated to his memory.  
\par The authors would also like to thank 
Karl Dilcher for communication on special primes and
pointing out reference \cite{KR} and 
Peter Stevenhagen for help with computing
the degree of the number field
$\mathbb Q(\zeta_{28}+\zeta_{28}^{-1},
\zeta_n,3^{1/n})$.


\end{document}